\documentclass[12pt]{amsart}

\usepackage[utf8]{inputenc}

\usepackage[utf8]{inputenc}
\usepackage{amssymb}
\usepackage{graphicx}
\usepackage{graphicx,color}
\usepackage{latexsym}
\usepackage[all]{xy}
\usepackage{mathrsfs}
\usepackage{enumerate}
\usepackage{multicol}
\usepackage{lipsum}
\usepackage{amssymb}
\usepackage{tikz}
\usepackage{float}
\usepackage{bm}
\usepackage{mathptmx}
\usetikzlibrary{shapes.geometric}

\usepackage{caption}
\usepackage{subcaption}

\usepackage{amsmath}
\usepackage{amsfonts}
\usepackage{amssymb,enumerate}
\usepackage{amsthm}
\usepackage{fancyhdr}
\usepackage{hyperref}
\usepackage{cleveref}
\usepackage{xcolor}
\usepackage{enumitem}
\usepackage{float}
\usepackage{verbatim}
\usepackage{tikz}
\usetikzlibrary[patterns]
\usetikzlibrary{arrows}
\usepackage{mathrsfs}
\usetikzlibrary{arrows}
\usepackage{listings}
\definecolor{DimGray}{rgb}{0.41, 0.41, 0.41}
\definecolor{zzttqq}{rgb}{0.6,0.2,0.}
\newtheorem{theorem}{Theorem}[section]
\theoremstyle{definition}
\newtheorem{defin}[theorem]{Definition}
\newtheorem{question}[theorem]{Question}

\newtheorem{ex}[theorem]{Example}
\theoremstyle{plain}
\newtheorem{proposition}[theorem]{Proposition}
\newtheorem*{theorem*}{Theorem}
\newtheorem{lemma}[theorem]{Lemma}
\newtheorem{corollary}[theorem]{Corollary}
\newtheorem{conj}[theorem]{Conjecture}
\theoremstyle{remark}
\newtheorem{rem}[theorem]{Remark}

\numberwithin{equation}{section}
\numberwithin{figure}{section}

\title[Supersymmetric gaps]{Supersymmetric gaps of a numerical semigroup with two generators}
\author{Patricio Almir\'on}
\author{Julio-Jos\'e Moyano-Fern\'andez}

\subjclass[2010]{Primary: 20M14; Secondary: 05A19,20M30.}

\keywords{Numerical semigroup, Frobenius problem, $\Gamma$-semimodule, syzygy}

\thanks{The first author was partially supported by Spanish Goverment, Ministerios de Ciencia e Innovaci\'on y de Universidades MTM2016-76868-C2-1-P. The second author was partially supported by the Spanish Government, Ministerios de Ciencia e Innovaci\'on y de Universidades, grant PGC2018-096446-B-C22, as well as by Universitat Jaume I, grant UJI-B2018-10. Both authors contributed equally to this work.}

\address{Instituto de Matemática interdisciplinar (IMI) y departamento de \'{A}lgebra, Geometr\'{i}a y Topolog\'{i}a\\
Facultad de Ciencias Matem\'{a}ticas\\
Universidad Complutense de Madrid\\
28040, Madrid, Spain.}

\email{palmiron@ucm.es}

\address{Universitat Jaume I, Campus de Riu Sec, Departamento de Matem\'aticas \& Institut Universitari de Matem\`atiques i Aplicacions de Castell\'o, 12071
Caste\-ll\'on de la Plana, Spain.}

\email{moyano@uji.es}

\begin{document}

\begin{abstract}
In this paper we introduce the new concepts of supersymmetric and self-symmetric gaps of a numerical semigroup with two generators. Those concepts are based on certain symmetries of the gaps of the semigroup with respect to their Wilf number. We prove that the set of supersymmetric and self-symmetric gaps completely determines the semigroup and we compare this set with the fundamental gaps of the semigroup.
\end{abstract}

\maketitle

\section{Introduction}
 A numerical semigroup $\Gamma$ is an additive submonoid of the monoid $(\mathbb{N},+)$ such that the greatest common divisor of all its elements is equal to $1$. The complement $\mathbb{N}\setminus \Gamma$ is therefore finite and the elements of that complement are called gaps of $\Gamma$. Moreover, $\Gamma$ is finitely generated and it is not difficult to find a minimal system of generators of $\Gamma$. 
\medskip

In 2004, Rosales et al. \cite{Rosetal} defined the concept of fundamental gaps of a numerical semigroup as an alternative way to represent a numerical semigroup as the set
 $$
 \mathcal{FG}=\mathcal{FG}(\Gamma):=\{g\in\mathbb{N}\setminus\Gamma:\;\;\{2g,3g\}\subset\Gamma\}.
 $$
 
 From \(\mathcal{FG}\) one can define the set
 \(\mathcal{D}(\mathcal{FG}):=\{x\in\mathbb{N}: \;x|x_i\;\text{for some }\;x_i\in\mathcal{FG}\}\) and see that \(\Gamma=\mathbb{N}\setminus \mathcal{D}(\mathcal{FG})\) provides an alternative representation of \(\Gamma\). Moreover, \(\mathcal{FG}\) is the smallest subset of \(\mathbb{N}\setminus\Gamma\) that \(H\)--determines the semigroup \(\Gamma\): a subset $X$ of $\mathbb{N}$ is said to $H$-determine $\Gamma$ if $\Gamma$ is the maximum (with respect to the set inclusion) numerical semigroup such that $X$ is a subset of $\mathbb{N}\setminus \Gamma$ (see \cite{Rosetal}; also Section \ref{subsec:fundvssym}).
 \medskip
 
In this paper, we are mainly concerned about numerical semigroups of the form \(\Gamma=\langle\alpha,\beta \rangle\). For those semigroups, we are going to introduce the concept of \emph{supersymmetric gap} and \emph{self-symmetric gap}. Let us denote $\mathsf{SG}$~resp.~$\mathsf{SSG}$ the set of supersymmetric~resp.~self-symmetric gaps. We prove that the set \(\mathsf{SG}\cup \mathsf{SSG}\) completely determines the semigroup \(\Gamma\) (see Theorem \ref{thm:partition}). Our construction lies on the application of certain affine linear transformations to the sets \(\mathsf{SG}, \mathsf{SSG}\) represented in the lattice \(\mathbb{N}^2\); we call the process of representation of \(\Gamma\) via those transformations \emph{polyomino game}. The set \(\mathsf{SG}\cup \mathsf{SSG}\) presents some differences over the set \(\mathcal{FG}\): it is not contained in \(\mathcal{FG}\) and its cardinality is less than or equal to the cardinality of the set of fundamental gaps whenever $\alpha >2$ (as well as in the case $\Gamma=\langle 2,3\rangle$), see Section \ref{subsec:fundvssym}. This is relevant, since \(\mathsf{SG}\cup \mathsf{SSG}\) is a sort of generating system of $\Gamma$ which is---in general---smaller than other existing systems.
\medskip

Our new concept of supersymmetric and self-symmetric gaps is closely related to a long standing conjecture proposed by H. Wilf in 1978 that can be formulated as follows \cite{wilf}:
 \begin{conj}[Wilf conjecture]
 	Let \(\Gamma\) be a numerical semigroup minimally generated by \(x_1,\dots,x_n\). Let us denote by \(c(\Gamma)\) the conductor of \( \Gamma\). Then,
 	\[
	c(\Gamma)\geq \frac{n}{n-1}|\mathbb{N}\setminus\Gamma|.
	\]
 \end{conj}

Inspired by this inequality, we consider a \(\Gamma\)--semimodule \(\Delta\) minimally generated by \(\mathrm{ed}(\Delta)\) elements  and we define the Wilf number of \(\Delta\) as
 	$$
	W(\Delta)=e(\Delta)\cdot \delta (\Delta)-c(\Delta),
	$$
where \(c(\Delta)\) denotes the conductor of \(\Delta\) and \(\delta (\Delta)=\{x\in \Delta : x < c(\Delta)\}.\) 
Observe that for $\Delta=\Gamma$ this number was already considered in \cite[p.~45]{mdelgado}.
\medskip

In the case \(\Gamma=\langle\alpha,\beta \rangle\), the second author together with Uliczka \cite{MU1} proved that every \(\Gamma\)--semimodule corresponds to a certain lattice path of the lattice \(\mathbb{N}^2\). The lattice path representation has already led to a formula for the conductor \(c(\Delta)\) of those \(\Gamma\)--semimodules, see \cite{almiyano}. In particular, this allows us to explicitly compute the Wilf number associated to a \(\Gamma\)--semimodule. Moreover, in the case that \(\Delta\) is generated by \([0,g]\) for \(g\in\mathbb{N}\setminus\Gamma\) we see that \(W(\Delta)\) only depends on \(g\) (see Proposition \ref{prop:wilfgap}) so in this case it will be referred to as the Wilf number associated to \(g\). Finally, we observe that the Wilf number provides a beautiful symmetry on the set of gaps (see Sect.~\ref{sec:supersym}) which motivates our definition of supersymmetric gaps. Moreover, if some of the generators of \(\Gamma\) is even, then there exist some gaps whose Wilf number vanishes and that are invariant under several operations; this motivates the name self-symmetric gaps (see Sections \ref{sec:wilf} and \ref{sec:supersym}).
 \medskip
 
To conclude, we discuss some issues regarding the possible extensions of the concepts of supersymmetric and self-symmetric gaps to the general case where \(\Gamma\) is a numerical semigroup with an arbitrary number of generators (see Subsection \ref{remarksconcepts}). More concretely, we propose a general definition for symmetric gaps (see Definition \ref{defin:extensionsym}) and we ask if this definition allows us to define the concepts of supersymmetric and self-symmetric gaps for a semigroup with any number of generators. We hope that---if this extension succeeds---these new concepts could be helpful to the solution of the Wilf conjecture. This hope is supported by some results in \cite[Sect.~4.1]{almiyano2}.

\vskip 2mm

\textbf{Acknowledgments.} The authors would like to thank Alfredo Granell Marqu\'es for the computation of general examples.

\section{Lattice paths and semimodules over a numerical semigroup}\label{sec:dos}

Let $\Gamma$ be a numerical semigroup. The reader is referred to \cite{RosalesGarciaSanchez} or \cite{ram} for specific material about numerical semigroups. We are interested in subsets of $\mathbb{N}$ which have an additive structure over $\Gamma$ (in analogy with the structure of module over a ring): a $\Gamma$-semimodule is a non-empty subset $\Delta$ of $\mathbb{N}$ with $\Delta+\Gamma\subseteq \Delta$; observe that these sets are also called \emph{relative ideals} in the literature, see e.g. \cite{KH}. A system of generators of $\Delta$ is a subset $\mathcal{E}$ of $\Delta$ with $\Delta=\bigcup_{x\in \mathcal{E}} (\Gamma + x)$; it is called minimal if no proper subset of $\mathcal{E}$ generates $\Delta$. Notice that, since $\Delta\setminus \Gamma$ is finite, every $\Gamma$-semimodule is finitely generated and for \(\Delta\neq \mathbb{N}\) has a conductor
$$
c(\Delta)=\max (\mathbb{N}\setminus \Delta)+1.
$$
To a semimodule $\Delta$ we also associate the \emph{$\delta$-invariant} of $\Delta$, namely
$$
\delta (\Delta)=\{x\in \Delta : x < c(\Delta)\}.
$$

Every $\Gamma$-semimodule $\Delta$ has a unique minimal system of generators (see e.g. \cite[Lemma 2.1]{MU1}); its cardinality is said to be the \emph{embedding dimension} of $\Delta$, written $\mathrm{ed}(\Delta)$. Two $\Gamma$-semimodules $\Delta$ and $\Delta'$ are called isomorphic if there is an integer $n$ such that $x\mapsto x+n$ is a bijection from $\Delta$ to $\Delta'$; we write then $\Delta\cong \Delta'$. For every $\Gamma$-semimodule $\Delta$ there is a unique semimodule $\Delta' \cong \Delta$ containing $0$; such a semimodule is called \emph{normalized}. The $\Gamma$-semimodule
$$
\Delta^{\circ} := \{x-\min \Delta : x \in \Delta\}
$$
is called the normalization of $\Delta$; this is the unique $\Gamma$-semimodule isomorphic to $\Delta$ and containing $0$. Moreover, the minimal system of generators $\{x_0=0,\ldots , x_n\}$ of a normalized $\Gamma$-semimodule is a $\Gamma$-lean set, i.e. it satisfies that
$$
|x_i-x_j| \notin \Gamma \ \ \mbox{for~any} \ \ 0\leq i <j \leq n,
$$
and conversely, every $\Gamma$-lean set of $\mathbb{N}$ minimally generates a normalized $\Gamma$-semimodule; we will write then $[x_0=0,\ldots , x_n]$. Hence there is a bijection between the set of isomorphism classes of $\Gamma$-semimodules and the set of $\Gamma$-lean sets of $\mathbb{N}$; see  Sect.~2 in \cite{MU1} for the proofs of those statements.
\medskip

The dual $\Delta^{\ast}$ of a $\Gamma$-semimodule $\Delta$ is defined to be
$$
\Delta^{\ast}:=\mathrm{Hom}_{\Gamma}(\Delta,\Gamma)=\{x\in \mathbb{N} : x+\Delta\subseteq \Gamma\},
$$
cf. \cite[p.~677]{MU2}. A $\Gamma$-semimodule is said to be \emph{selfdual} if $\Delta=\Delta^{\ast}$. In addition, we define the set of syzygies of a $\Gamma$-semimodule $\Delta=\Delta_I$ with minimal set of generators $I=[g_0,\ldots , g_n]$ as
$$
\mathrm{Syz}(\Delta):=\bigcup_{i,j\in I, i\neq j} \Big ((\Gamma + g_i)\cap (\Gamma + g_j) \Big ).
$$

\medskip

\subsection*{Numerical semigroups with embedding dimension 2} In this paper we will consider numerical semigroups with two generators, say $\Gamma=\langle \alpha, \beta \rangle$, with $\alpha,\beta \in \mathbb{N}$ and $\alpha < \beta $.
The conductor of $\Gamma$ can be expressed as $c=c(\Gamma)=(\alpha-1)(\beta-1)$. The gaps of $\langle \alpha, \beta \rangle$ are also easy to describe: they admit a representation $\alpha \beta -a\alpha -b \beta$, where $a\in \ ]0,\beta-1]\cap \mathbb{N}$ and $b\in \ ]0,\alpha-1]\cap \mathbb{N}$, see Rosales \cite[Lemma~1]{Rosales}. This writing yields a map from the set of gaps of $\langle \alpha, \beta \rangle$ to $\mathbb{N}^2$ given by $\alpha \beta -a\alpha -b \beta \mapsto (a,b)$, which allows us to identify a gap with a point in the lattice $\mathcal{L}=\mathbb{N}^2$; since the gaps are positive numbers, the point lies inside the triangle with vertices $(0,0),(0,\alpha),(\beta, 0)$. Let us denote by \(LG\) the image of the map $\alpha \beta -a\alpha -b \beta \mapsto (a,b)$, i.e. the points of \(\mathcal{L}\) inside the triangle of vertices $(0,0),(0,\alpha),(\beta, 0)$.
\medskip

In the following we will use the notation
$$
e=\alpha\beta - a(e)\alpha-b(e)\beta
$$
for a gap $e$ of the semigroup $\langle \alpha,\beta \rangle$; if the gap is subscripted as $e_i$ then we write $a_i=a(e_i)$ and $b_i=b(e_i)$. 
\medskip

Let us denote by \(\leq\) the total ordering in \(\mathbb{N},\) if needed we will denote it by \(\leq_{\mathbb{N}}\) to emphasize that it is the natural order. We also consider the following partial ordering $\preceq$ on the set of gaps:

\begin{defin}\label{jorder}
	Given two gaps $e_1,e_2$ of $\langle \alpha , \beta \rangle$, we define 
	$$
	e_1 \preceq e_2 \ \ :\Longleftrightarrow \ a_1\leq a_2 \ \ \wedge \ \ b_1 \geq b_2
	$$
	and
	$$
	e_1 \prec e_2 \ \  :\Longleftrightarrow  \ a_1 < a_2 \ \ \wedge \ \ b_1 >b_2.
	$$
\end{defin}

Let $\mathcal{E}=\{0,e_1,\ldots , e_n\}$ be a subset of $\mathbb{N}$ with gaps $e_i=\alpha\beta -a_i \alpha -b_i\beta$ of $\langle \alpha, \beta \rangle$ for every $i=1,\ldots, n$ such that $a_1<a_2<\cdots < a_n$. Corollary 3.3 in \cite{MU1} ensures that $\mathcal{E}$ is $\langle \alpha , \beta \rangle$-lean if and only if $b_1>b_2>\cdots > b_n$. This simple fact leads to an identification (cf. \cite[Lemma 3.4]{MU1}) between an $\langle \alpha, \beta \rangle$-lean set and a lattice path with steps downwards and to the right from $(0,\alpha)$ to $(\beta,0)$ not crossing the line joining these two points, where the lattice points identified with the gaps in $\mathcal{E}$ mark the turns from the $x$-direction to the $y$-direction; these turns will be called ES-turns for abbreviation. Figure \ref{fig:lattice1} shows the lattice path corresponding to the $\langle 5,7 \rangle$-lean set $[0,9,11,8]$.

\begin{center}
	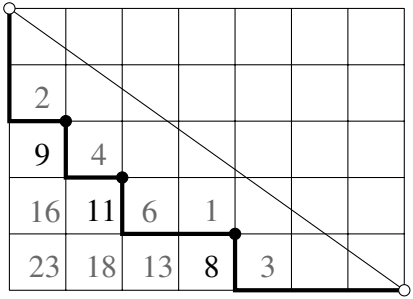
\begin{figure}[H]
		\begin{tikzpicture}[scale=0.75]
		\draw[] (0,0) grid [step=1cm](7,5);
		\draw[] (0,5) -- (7,0);
		\draw[ultra thick] (0,5) -- (0,3) -- (1,3) -- (1,2) -- (2,2) -- (2,1) -- (4,1) -- (4,0) -- (7,0);
		\draw[fill=white] (0,5) circle [radius=0.1]; 
		\draw[fill] (1,3) circle [radius=0.1]; 
		\draw[fill] (2,2) circle [radius=0.1]; 
		\draw[fill] (4,1) circle [radius=0.1]; 
		\draw[fill=white] (7,0) circle [radius=0.1]; 
		
		\node [below right][DimGray] at (0.15,0.8) {$23$};
		\node [below right][DimGray] at (1.15,0.8) {$18$};
		\node [below right][DimGray] at (2.15,0.8) {$13$};
		\node [below right] at (3.25,0.8) {$8$};
		\node [below right][DimGray] at (4.25,0.8) {$3$};
		
		\node [below right][DimGray] at (0.15,1.8) {$16$};
		\node [below right] at (1.15,1.8) {$11$};
		\node [below right][DimGray] at (2.15,1.8) {$6$};
		\node [below right][DimGray] at (3.25,1.8) {$1$};
		
		\node [below right] at (0.25,2.8) {$9$};
		\node [below right][DimGray] at (1.25,2.8) {$4$};
		
		\node [below right][DimGray] at (0.25,3.8) {$2$};
		\end{tikzpicture}
		\caption{Lattice path for the $\langle 5,7 \rangle$-lean set $[0,9,11,8]$.} \label{fig:lattice1}
	\end{figure}
\end{center}
\vspace{-0.9cm}

Let $g_0=0,g_1,\ldots , g_n$ be the minimal system of generators of a $\langle \alpha, \beta \rangle$-semimodule $\Delta$. From now on, we will assume that the indexing in the minimal set of generators of \(\Delta\) is such that \(g_0=0\preceq g_1\preceq\cdots\preceq g_n.\) Under this assumption, we can give an explicit formula for the minimal generators of $\Delta^{\ast}$ in terms of those of $\Delta$: 

\begin{equation}\label{eqn:dual}
\Delta^{\ast}=(\Gamma + a_1\alpha) \cup \bigcup_{k=1}^{n-1} (\Gamma+a_{k+1}\alpha+b_k\beta) \cup (\Gamma+b_n\beta).
\end{equation}

Moreover, the semimodule $\mathrm{Syz}(\Delta)$ of syzygies of $\Delta$ can be characterized as follows (see \cite[Theorem 4.3]{MU1}):
\begin{proposition}\label{defin:syz}
	\begin{equation*}
	\mathrm{Syz}(\Delta) =\bigcup_{0\leq k<j\leq n} \Big ((\Gamma + g_k)\cap (\Gamma + g_j) \Big )= \bigcup_{k=0}^{n} (\Gamma + h_k),
	\end{equation*}
	where $h_1,\ldots , h_{n-1}$ are gaps of $\Gamma$, $h_0,h_n \leq \alpha \beta$, and
	\begin{align*}
	&h_k \equiv g_k ~\mathrm{ mod }~ \alpha, ~h_ k > g_{k}  \ \mbox{for } k=0,\ldots , n \\
	&h_k \equiv g_{k+1} ~ \mathrm{mod } ~\beta,~ h_ k > g_{k+1} \ \mbox{for } k=0,\ldots , n-1\\
	&h_n \equiv 0~\mathrm{mod } ~\beta, \ \mathrm{and } ~h_n \geq 0
	\end{align*}
\end{proposition} 

In particular, \(J=[h_0,\dots,h_n]\) is a minimal system of generators of the semimodule \(\Delta_J=\mathrm{Syz}(\Delta)\), hence \(h_0\preceq h_1\preceq\cdots\preceq h_n.\) Therefore it is easily seen that the SE-turns of the lattice path associated to \(\Delta\) can be identified with the minimal set of generators of the syzygy module (we call SE-turns to the turns from the \(y\)--direction to the \(x\)--direction). After that, we can associate to any \(\Gamma\)-semimodule \(\Delta\) a lean couple \((I,J)\) where \(I\) is a minimal set of generators of \(\Delta\) and \(J\) a minimal set of generators of \(\mathrm{Syz}(\Delta);\) or equivalently a lattice path. The syzygies allowed us to give a formula for the conductor $c(\Delta)$ of $\Delta$:

\begin{theorem}\cite[Theorem 1]{almiyano}\label{formula}
	Let \(\Delta\) be a \(\Gamma\)-semimodule with $\Gamma$-lean couple \((I,J)\), and let \(M:=\max_{\leq_\mathbb{N}}\{h\in J\}\) denote the biggest, with respect to the order of the natural numbers, minimal generator of the syzygy module of $\Delta$. Then
	\[c(\Delta)=M-\alpha-\beta+1.\]
	In particular, if we denote by \((m_1,m_2)\) the point in the lattice $\mathcal{L}$ representing \(M\), then we have
	\[c(\Delta)=c(\Gamma)-m_1\alpha-m_2\beta.\]
\end{theorem}

The syzygies lead also to the concept of \emph{fixed point} for a semimodule:

\begin{defin}
An $\langle \alpha,\beta \rangle$-semimodule $\Delta_I$ with associated $\Gamma$-lean couple $(I,J)$ is said to be a $\langle \alpha, \beta\rangle$-fixed point (or simply a fixed point if the semigroup is clear from the context) if the semimodule $(\Delta_J)^{\circ}$ admits $I$ again as a minimal system of generators. 
\end{defin}

The chosen name \emph{fixed point} has a reason: it refers to the orbits of period $1$ of the Picard sequence associated to the map $f=h\circ \mathrm{Syz}$, where $\mathrm{Syz}$ is the map $\Delta_I\mapsto \Delta_J$ and $h$ is the normalization map for $\Delta_J=\mathrm{Syz}(\Delta_I)$; this is further explained in \cite[Sect.~5]{MU1}.


\section{Wilf number of a gap of a numerical semigroup}\label{sec:wilf}

In this section, we are going to make use of the conductor formula of a \(\Gamma\)--semimodule (Theorem \ref{formula}) to associate to a gap a number which will be invariant under certain symmetries of the lattice. This number is motivated by Wilf's conjecture. First, we introduce the \emph{Wilf number} of a $\Gamma$-semimodule:

\begin{defin}
	Let $\Delta$ be a $\Gamma$-semimodule, then the Wilf number of $\Delta$ is defined to be
		$$
	W(\Delta)=e(\Delta)\cdot \delta (\Delta)-c(\Delta).
	$$
	In addition, we define the Wilf number of a gap $g\in \mathbb{N}\setminus\Gamma$ by assigning the Wilf number of the $\Gamma$-semimodule $\Delta=\Delta_I$ minimally generated by $I=[0,g]$:

$$
W(g):=W(\Delta_I)=2\delta(\Delta_{I})-c(\Delta_{I}).
$$
	
	If \(g=\alpha\beta-a\alpha-b\beta\) then we will also denote by \(W(a,b):=W(g)\) its Wilf number.
\end{defin}

Observe that for $\Delta=\Gamma$, Wilf's conjecture claims that $W(\Gamma)\geq 0$, see \cite{mdelgado,wilf}. 
\medskip

Now, we restrict our attention to the case $\Gamma=\langle \alpha, \beta \rangle$. Here $W(g)$  only depends on the gap $g$ as a consequence of the formula for the conductor of a \(\Gamma\)--semimodule given in Theorem \ref{formula}.
\begin{proposition}\label{lem:formuladelta}
Let $I=[g_0=0,g_1,\ldots, g_n]$ be the minimal system of generators of a $\Gamma$-semimodule $\Delta$ ordered as $0\prec g_1\prec \cdots \prec g_n$; recall the writing $g_{i}=\alpha\beta -a_{i}\alpha-b_{i}\beta$, and set $a_0=b_0=0$. Then
\begin{align*}
\delta(\Delta)=&c(\Delta)-\delta(\Gamma)+\sum_{i=0}^{n} (a_{i+1}-a_i)b_{i+1}\\
=& c(\Delta)-\delta(\Gamma)+\sum_{i=0}^{n} (b_{i}-b_{i+1})a_{i+1}.
\end{align*}
Furthermore, for any $I=[0,g_i]$ with $i=1,\ldots , n$, we have
$$
\delta(\Delta_{I})=c(\Delta_{I})-\delta(\Gamma)+a_ib_i.
$$
\end{proposition}
\begin{proof}
	Since \(\delta(\Gamma)=|\mathbb{N}\setminus \Gamma|\), it is easily deduced from the lattice path representation of \(\Delta_{I}\) that 
	\[
	|\mathbb{N}\setminus\Delta_{I}|=\delta(\Gamma)-\sum_{i=0}^{n} (a_{i+1}-a_i)b_{i+1}=\sum_{i=0}^{n} (b_{i}-b_{i+1})a_{i+1}.
	\]
	By the definition of \(\delta(\Delta_{I})\) we have the claim.
\end{proof}

Therefore, we can compute explicitly the Wilf number of a gap of $\langle \alpha, \beta \rangle$:

\begin{proposition}\label{prop:wilfgap}
	Let \(g=\alpha\beta-a\alpha-b\beta\) be a gap of $\Gamma$. Let us denote by \([h_0,h_1]\) the minimal system of generators of the $\Gamma$-semimodule \(\mathrm{Syz}(\Delta_{[0,g]})\). Then
	\[-W(g)=\Big\{\begin{array}{cc}
	a\alpha-2ab&\text{if}\;\;\;\mathrm{min}\{h_0,h_1\}=\alpha\beta-b\beta,\\
	b\beta-2ab&\text{if}\;\;\;\mathrm{min}\{h_0,h_1\}=\alpha\beta-a\alpha.
	\end{array}\]
\end{proposition}
\begin{proof}
	Consider the \(\Gamma\)--semimodule $\Delta_I$ generated by \(I=[0,g]\). From the representation of the lattice path we know that \(h_0=\alpha\beta-b\beta\) and \(h_1=\alpha\beta-a\alpha\). Let us first assume \(\mathrm{min}\{h_0,h_1\}=h_0\), thus \(\mathrm{max}\{h_0,h_1\}=h_1=\alpha\beta-a\alpha\). Therefore, by Theorem \ref{formula} we have
	\[c(\Delta_{I})=\alpha\beta-a\alpha-\alpha-\beta+1.\]
	Proposition \ref{lem:formuladelta} yields
	\[c(\Delta_{I})-2\delta(\Delta_{I})=-c(\Delta_{I})+2\delta(\Gamma)-2ab=c(\Gamma)-c(\Delta_{I})-2ab=a\alpha-2ab.\]
	The same reasoning applies after exchanging the roles of \(a\) and \(b\) to obtain the second case of the formula.
\end{proof}

We are interested in the case in which the Wilf number of a gap is zero:

\begin{theorem}\label{thm:equifix}
	Let $\Gamma=\langle \alpha , \beta \rangle$ and let $g=\alpha\beta - a \alpha - b \beta$ be a gap of $\Gamma$. Consider the $\Gamma$-semimodule $\Delta=\Delta_I$ minimally generated by $I=[0,g]$. The following statements are equivalent:
	\begin{enumerate}
		\item $W(g)=0$;
		\item either $\alpha=2b$ or $\beta=2a$;
		\item $\Delta$ is a fixed point;
		\item $\Delta$ is selfdual;
		\item $\Delta$ is symmetric, i.e. for every $x\in \Delta$ we have that $c(\Delta)-1-x\notin \Delta$ if and only if $x\in \Delta$.
	\end{enumerate}
\end{theorem}

\begin{proof}
	$(1) \Longleftrightarrow (2)$ is obvious by Proposition \ref{prop:wilfgap}.\\
	$(2) \Longleftrightarrow (3)$: If $\alpha=2b$, then $g=\alpha\beta-a\alpha-b\beta=b\beta-a\alpha=h_1-h_0$, and this is positive since $g$ is a gap, hence $\mathrm{Syz}(\Delta_{[0,g]})=\Delta_{[h_0,h_1]}=\Delta_{[0,h_1-h_0]}=\Delta_{[0,g]}$. \emph{Mutatis mutandis}, if $\beta=2a$, then $\Delta$ is a fixed point. Conversely, assume without loss of generality that $h_0<h_1$; since $\alpha\beta-a\alpha-b\beta= g=h_1-h_0=b\beta-a\alpha$, it is easily seen that $\alpha=2b$.\\
	$(3) \Longleftrightarrow (4)$: First we observe that the number of semimodules $\Delta_I$ with $I=[0,g]$ which are fixed points coincide with the number of selfdual modules of that form, as a direct application of Theorem 5.5 in \cite[Theorem 5.5]{MU1}, as well as Proposition 4.1 and Theorem 4.4 in \cite{MU2}. Moreover, every selfdual semimodule is a fixed point: For $\Delta=\Delta_I$ with $I=[0,g]$, eq.~(\ref{eqn:dual}) implies that $\Delta^{\ast}$ is minimally generated by $[a_1\alpha,b_1\beta]$; the selfduality implies that $a_1=0$ and $b_1\beta=\alpha\beta-b_1\beta$, hence $I=[0,\alpha\beta-b_1\beta]$. On the other hand, $\mathrm{Syz}(\Delta_I)$ is minimally generated by $[a_1\alpha,b_1\beta]=[0,\alpha\beta-b_1\beta]$, that equals its own normalization. Therefore $\Delta_I$ coincides with the normalization of $\mathrm{Syz}(\Delta_I)$ and $\Delta_I$ is a fixed point.\\
	$(4) \Longleftrightarrow (5)$: this is a consequence of Proposition 4 in \cite{tozzo}, and Theorem 2.11 together with Proposition 3.8 in \cite{moyano}.
\end{proof}

For $\langle \alpha, \beta \rangle$-semimodules $\Delta$ with $\mathrm{ed}(\Delta)>2$, Theorem \ref{thm:equifix} is no longer true: for instance the $\langle 5,8\rangle$-semimodule minimally generated by the lean set $I=[0, 4, 6, 7]$ has Wilf number $W(\Delta_I)=4\delta(\Delta_I)-c(\Delta_I)=4\cdot 1 -4=0$ and it is not a fixed point. Moreover, let us consider the numerical semigroup $\Gamma=\langle 10,14,27\rangle$. This semigroup is both symmetric and complete intersection, however if we consider the \(\Gamma\)--semimodule generated by \(I=[0,9]\) then \(W(\Delta_{I})=0\) and \(I=[0,9]\) is neither a fixed point nor symmetric. On the other hand, if we consider $\Gamma=\langle 10,14,29\rangle$ every \(\Gamma\)--semimodule with Wilf number equal to zero is a fixed point. 

Therefore, we cannot expect a generalization of Theorem \ref{thm:equifix} for numerical semigroups \(\Gamma\) with more than two minimal generators just by imposing the condition of symmetric or complete intersection. This encourages us to propose the following question.



\begin{question}\label{quest:zerosym}
Given a numerical semigroup $\Gamma$ with $\mathrm{ed}(\Gamma)>2$, does there exist a class of numerical semigroups for which any of the equivalences of Theorem \ref{thm:equifix} remain true?
\end{question}

\section{Supersymmetry of the gap set with respect to the Wilf number}\label{sec:supersym}

This section is devoted to introduce and to develop the concept of supersymmetric gaps. This notion is based on certain symmetries encoded in the Wilf number of a gap. Before introducing those concepts, let us see some properties of the image \(LG\) of the map $\alpha\beta -a\alpha-b\beta \mapsto (a,b)$, cf. Sect. \ref{sec:dos}.
\medskip

\subsection{Determinacy of the semigroup}
A first observation is that a \(\Gamma\)--semimodule generated by \([0,g]\) has its syzygy module with two minimal generators. Moreover, by the formulas for the minimal set of generators of the syzygy module the following lemma is a straightforward computation.
\begin{lemma}\label{lem:minsyz}
	Let \(g=\alpha\beta-a\alpha-b\beta \) be a gap of \(\Gamma=\langle\alpha,\beta\rangle\). Let \([h_0,h_1]\) be the minimal set of generators of \(\mathrm{Syz}(\Delta_{[0,g]})\). Then,
	\begin{enumerate}
		\item If \(b>\lfloor\frac{\alpha}{2}\rfloor\) and \(a\leq\lfloor\frac{\beta}{2}\rfloor\) then \(\mathrm{min}\{h_0,h_1\}=\alpha\beta-b\beta.\)
		\item If \(b\leq\lfloor\frac{\alpha}{2}\rfloor\) and \(a>\lfloor\frac{\beta}{2}\rfloor\) then \(\mathrm{min}\{h_0,h_1\}=\alpha\beta-a\alpha.\)
	\end{enumerate}
\end{lemma}
This lemma allows us to describe the behavior of the set of gaps with respect to the Wilf number. First, observe that any integral point inside the triangle delimited by the \(y\)--axis, the line \(y=\lfloor\frac{\alpha}{2}\rfloor\) and the diagonal \(\alpha\beta=x\alpha+y\beta\) represents a gap \(g\) of \(\Gamma\) with expression \(g=\alpha\beta-a\alpha-b\beta\) and \(b>\lfloor\frac{\alpha}{2}\rfloor\). Hence this gap has Wilf number \(-W(g)=a\alpha-2ab\). Now, let us consider the symmetric point to \(g\) with respect to the reflection along the line \(y=\lfloor\frac{\alpha}{2}\rfloor\). This reflection is given by the map \((a,b)\mapsto(a,\alpha-b)\). Therefore, we have 
\begin{lemma}\label{lem:sym1}
	If \((a,b)\) is an integral point inside the triangle delimited by the \(y\)--axis, the line \(y=\lfloor\frac{\alpha}{2}\rfloor\) and the diagonal \(\alpha\beta=x\alpha+y\beta\) then
	
	\[
	-W(a,b)=W(a,\alpha-b).
	\]
\end{lemma}
\begin{proof}
	By Lemma \ref{lem:minsyz} and Proposition \ref{prop:wilfgap} we have that \(-W(a,b)=a\alpha-2ab\). Now, let us denote \(g_{\mathrm{sym}}=\alpha\beta-a\alpha-(\alpha-b)\beta\) the symmetric gap with respect to the reflection \((a,b)\mapsto(a,\alpha-b)\). Let us consider the minimal set of generators \([h'_{0},h'_{1}]\) of \(\mathrm{Syz}(\Delta_{[0,g_{\mathrm{sym}}]})\). It is thus clear that \(\mathrm{min}\{h'_{0},h'_{1}\}=\alpha\beta-b\beta\), since
	
	 \[h'_{1}-h'_{0}=\alpha\beta-a\alpha-\alpha\beta+(\alpha-b)\beta=\alpha\beta-a\alpha-b\beta=g>0,\]
	and the proof follows.
\end{proof}

An analogous situation occurs when considering the triangle delimited by the \(x\)--axis, the line \(x=\lfloor\frac{\beta}{2}\rfloor\) and the diagonal \(\alpha\beta=x\alpha+y\beta\). In this case, the map \((a,b)\mapsto(\beta-a,b)\) yields the following result:
\begin{lemma}\label{lem:sym2}
	If \((a,b)\) is an integral point inside the triangle delimited by the \(x\)--axis, the line \(x=\lfloor\frac{\beta}{2}\rfloor\) and the diagonal \(\alpha\beta=x\alpha+y\beta\), then

	\[
	-W(a,b)=W(\beta-a,b).
	\]
\end{lemma}
In particular, the set of fixed points of each of the previous symmetries is exactly the set of points with \(W(a,b)=0\). As we have seen in Theorem \ref{thm:equifix} those are exactly fixed points of the orbits of the associated lattice path, i.e. of the associated semimodule.
\medskip

The previous discussion leads to the following definition:
\begin{defin}\label{def:supersym}
	Let \(\Gamma=\langle\alpha,\beta\rangle\) be a numerical semigroup.  Let us denote by \(\mathcal{T}_{r}\) the set of points of \(\mathcal{L}\) inside (and not in the border of) the triangle delimited by the \(x\)--axis, the line \(x=\lfloor\frac{\beta}{2}\rfloor\) and the diagonal \(\alpha\beta=x\alpha+y\beta\) and  \(\mathcal{T}_{u}\) the set of points of \(\mathcal{L}\) inside (and not in the border of) the triangle delimited by the \(y\)--axis, the line \(y=\lfloor\frac{\alpha}{2}\rfloor\) and the diagonal \(\alpha\beta=x\alpha+y\beta\). The set of \emph{supersymmetric gaps} is defined to be
	\[\mathsf{SG}:=\Big\{\begin{array}{cc}
	\mathcal{T}_{u}&\text{if}\;\;\;|\mathcal{T}_{u}|<|\mathcal{T}_{r}|\\
		\mathcal{T}_{r}&\text{if}\;\;\;|\mathcal{T}_{r}|<|\mathcal{T}_{u}|.
	\end{array}\]
	We also define the set of \emph{self-symmetric gaps}
	\[\mathsf{SSG}:=\{g\in \mathbb{N}\setminus\Gamma : \;W(g)=0\}.\]
\end{defin}

\begin{ex}

For the semigroup $\Gamma=\langle 7,8 \rangle$ we have that the set \(\mathcal{T}_{r}\) consists of the gaps $5,13,6$, and the set \(\mathcal{T}_{u}\) is made up with the gaps $1,9,17,2,10,3$. Hence $\mathsf{SG}=\mathcal{T}_{r}$, this is represented in the shaded part of Figure \ref{fig:xz}. It is easily checked that $\mathsf{SSG}=\{4,12,20\}$. This is represented in the striped part of Figure \ref{fig:xz}.

		\begin{figure}[H]
			\begin{center}
			
			\begin{tikzpicture}[line cap=round,line join=round,>=triangle 45,x=0.6cm,y=0.6cm]
			\clip(-0.5,-0.5) rectangle (8.5,7.5);
			
			\fill[line width=2.pt,color=zzttqq,fill=zzttqq,pattern=north east lines,pattern color=zzttqq] (3.,3.) -- (4.,3.) -- (4.,0.) -- (3.,0.) -- cycle;
			\fill[line width=2.pt,color=zzttqq,fill=zzttqq,fill opacity=0.20] (4.,2.) -- (5.,2.) -- (5.,1.) -- (6.,1.) -- (6.,0.) -- (4.,0.) -- cycle;
			\draw [line width=1.pt,dash pattern=on 3pt off 3pt] (0.,7.)-- (0.,0.);
			\draw [line width=1.pt,dash pattern=on 3pt off 3pt]  (0.,0.)-- (8.,0.);
			\draw [line width=1.pt,dash pattern=on 3pt off 3pt]  (8.,0.)-- (8.,7.);
			\draw[line width=1.pt,dash pattern=on 3pt off 3pt]  (8.,7.)-- (0.,7.);
			\draw [line width=1.pt,dash pattern=on 3pt off 3pt]  (1.,7.)-- (1.,0.);
			\draw [line width=1.pt,dash pattern=on 3pt off 3pt]  (2.,0.)-- (2.,7.);
			\draw [line width=1.pt,dash pattern=on 3pt off 3pt]  (3.,0.)-- (3.,7.);
			\draw[line width=1.pt,dash pattern=on 3pt off 3pt]  (4.,7.)-- (4.,0.);
			\draw[line width=1.pt,dash pattern=on 3pt off 3pt]  (5.,0.)-- (5.,7.);
			\draw [line width=1.pt,dash pattern=on 3pt off 3pt]  (6.,7.)-- (6.,0.);
			\draw [line width=1.pt,dash pattern=on 3pt off 3pt]  (7.,0.)-- (7.,7.);
			\draw[line width=1.pt,dash pattern=on 3pt off 3pt]  (0.,1.)-- (8.,1.);
			\draw [line width=1.pt,dash pattern=on 3pt off 3pt]  (8.,2.)-- (0.,2.);
			\draw[line width=1.pt,dash pattern=on 3pt off 3pt]  (0.,3.)-- (8.,3.);
			\draw [line width=1.pt,dash pattern=on 3pt off 3pt] (8.,4.)-- (0.,4.);
			\draw[line width=1.pt,dash pattern=on 3pt off 3pt]  (0.,5.)-- (8.,5.);
			\draw [line width=1.pt,dash pattern=on 3pt off 3pt]  (8.,6.)-- (0.,6.);
			\draw[line width=1.5pt]  (0.,7.)-- (8.,0.);
			\draw [line width=1.2pt,dash pattern=on 3pt off 3pt] (0.,3.)-- (3.,3.);
			\draw [line width=1.2pt,dash pattern=on 3pt off 3pt] (3.,3.)-- (3.,2.);
			\draw [line width=1.2pt,dash pattern=on 3pt off 3pt] (3.,2.)-- (2.,2.);
			\draw [line width=1.2pt, dash pattern=on 3pt off 3pt] (2.,2.)-- (2.,1.);
			\draw [line width=1.2pt, dash pattern=on 3pt off 3pt] (2.,1.)-- (1.,1.);
			\draw [line width=1.2pt, dash pattern=on 3pt off 3pt] (1.,1.)-- (1.,0.);
			\draw [line width=1.2pt, dash pattern=on 3pt off 3pt] (1.,0.)-- (0.,0.);
			\draw [line width=1.2pt, dash pattern=on 3pt off 3pt] (0.,0.)-- (0.,3.);
			\draw [line width=1.2pt, dash pattern=on 3pt off 3pt] (0.,3.)-- (3.,3.);
			\draw [line width=1.2pt, dash pattern=on 3pt off 3pt] (3.,3.)-- (3.,4.);
			\draw [line width=1.2pt, dash pattern=on 3pt off 3pt] (3.,4.)-- (2.,4.);
			\draw [line width=1.2pt, dash pattern=on 3pt off 3pt] (2.,4.)-- (2.,5.);
			\draw [line width=1.2pt, dash pattern=on 3pt off 3pt] (2.,5.)-- (1.,5.);
			\draw [line width=1.2pt, dash pattern=on 3pt off 3pt] (1.,5.)-- (1.,6.);
			\draw [line width=1.2pt, dash pattern=on 3pt off 3pt] (1.,6.)-- (0.,6.);
			\draw [line width=1.2pt, dash pattern=on 3pt off 3pt] (0.,6.)-- (0.,3.);
			
			\draw (0,0.9) node[anchor=north west] {\small 41};
			\draw (0,1.9) node[anchor=north west] {\small 33};
			\draw (0,2.9) node[anchor=north west] {\small 25};
			\draw (0,3.9) node[anchor=north west] {\small 17};
			\draw (0,4.9) node[anchor=north west] {\small 9};
			\draw (0,5.9) node[anchor=north west] {\small 1};
			
			\draw (1,1.9) node[anchor=north west] {\small 26};
			\draw (1,2.9) node[anchor=north west] {\small 18};
			\draw (1,3.9) node[anchor=north west] {\small 10};
			\draw (1,4.9) node[anchor=north west] {\small 2};
			
			\draw (2,2.9) node[anchor=north west] {\small 11};
			\draw (2,3.9) node[anchor=north west] {\small 3};
			
			\draw (4,0.9) node[anchor=north west] {\small 13};
			\draw (5,0.9) node[anchor=north west] {\small 6};
			\draw (4,1.9) node[anchor=north west] {\small 5};
			
			\draw (1,0.9) node[anchor=north west] {\small 34};
			\draw (2,0.9) node[anchor=north west] {\small 27};
			\draw (3,0.9) node[anchor=north west] {\small 20};
			
			\draw (2,1.9) node[anchor=north west] {\small 19};
			\draw (3,1.9) node[anchor=north west] {\small 12};
			
			\draw (3,2.95) node[anchor=north west] {\small 4};
			\end{tikzpicture}
			\caption{Lattice representation of the gap set \(\mathbb{N}\setminus\Gamma\). The shaded set is \(\mathsf{SG}\) and the striped one is \(\mathsf{SSG}\). }
			\label{fig:xz}
			
		\end{center}
		\end{figure}
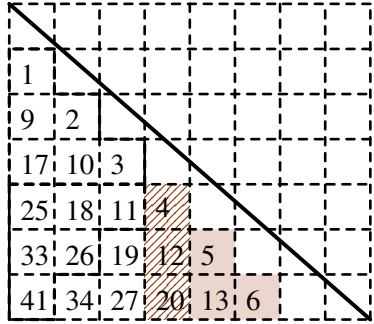
\end{ex}

At this point, we are able to prove the main result of the paper.
\begin{theorem}\label{thm:partition}
	Let \(\Gamma=\langle\alpha,\beta\rangle\) be a numerical semigroup. Then the set \(\mathsf{SG}\cup \mathsf{SSG}\) of supersymmetric and self-symmetric gaps completely determines the set of gaps of \(\Gamma\). In particular, it determines \(\Gamma\) itself.
\end{theorem}
\begin{proof}
	With the notation of Definition \ref{def:supersym}, consider the symmetry \(s_{\alpha}:\mathcal{T}_u\rightarrow LG\) along the line \(y=\lfloor\frac{\alpha}{2}\rfloor\) defined by \((a,b)\mapsto(a,\alpha-b)\), as well as the symmetry \(s_{\beta}:\mathcal{T}_r\rightarrow LG\) along the line \(x=\lfloor\frac{\beta}{2}\rfloor\) defined by \((a,b)\mapsto(\beta-a,b)\).
	First we are going to show that \(s_{\alpha}(\mathcal{T}_u)\cap s_{\beta}(\mathcal{T}_r)=\emptyset.\) Consider \((a,b)\in\mathcal{T}_u\) then
	
	\[
	(a,b)\mapsto(a,\alpha-b)\mapsto(\beta-a,\alpha-b),
	\]
	where \(\alpha\beta-(\beta-a)\alpha-(\alpha-b)\beta=a\alpha+b\beta-\alpha\beta<0\), since \(\alpha\beta-a\alpha-b\beta\) is the representation of a gap. Therefore, \(s^{-1}_{\beta}(s_{\alpha}(\mathcal{T}_u))=\emptyset\). Analogously, it can be shown that \(s^{-1}_{\alpha}(s_{\beta}(\mathcal{T}_u))=\emptyset\). 
	
	\medskip
	Now, let $B(\mathcal{T}_u)$ resp.~$B(\mathcal{T}_r)$ be the border points of sets $\mathcal{T}_u$~resp.~$\mathcal{T}_r$, i.e. those points such that \((a,b)\in \mathcal{T}_u\)~resp. \(\mathcal{T}_r\)~and \((a+1,b)\notin LG\) or \((a,b+1)\notin LG\). Moreover, let $RB(\mathcal{T}_u)$ resp.~$RB(\mathcal{T}_r)$ denote the set of border points of the type \((a+1,b)\notin LG\). Observe that those points determine ES--turns, hence the borders $B(\mathcal{T}_u)$ and $B(\mathcal{T}_r)$ are determined by $RB(\mathcal{T}_u)$ and $RB(\mathcal{T}_r)$.
	
	\medskip
	Let us denote by \(\tau:\mathcal{L}\rightarrow\mathcal{L}\) the translation defined by \((a,b)\mapsto(a+1,b)\). We claim that
	
	\[
	s^{-1}_{\beta}(\tau(s_{\alpha}(RB(\mathcal{T}_u))))=RB(\mathcal{T}_r).
	\]
	Indeed, consider the point \((a,b)\in\mathcal{T}_u\), then \(s^{-1}_{\beta}(\tau(s_{\alpha}((a,b))))=(\beta-a-1,\alpha-a)\in RB(\mathcal{T}_r)\) due to the fact that $\alpha\beta-(\beta-a-1)\alpha-b\beta>0$ and $\alpha\beta-(\beta-a)\alpha-b\beta<0$.  A similar reasoning allows us to prove the equality
	
	\[s^{-1}_{\alpha}(\tau^{-1}(s_{\beta}(RB(\mathcal{T}_r))))=RB(\mathcal{T}_u).\]
	The proof will finish by distinguishing three cases concerning the parity of $\alpha$ and $\beta$. Let us start with easiest one and assume that \(\alpha,\beta\) are both odd. By Theorem \ref{thm:equifix} there are no gaps with \(W(g)=0\). Thus, we have a configuration as in Figure \ref{fig:proofwithoutgaps}.
	\begin{figure}[H]
		\begin{center}
		\begin{tikzpicture}[line cap=round,line join=round,>=triangle 45,x=0.7cm,y=0.7cm]
		\clip(-0.5,-0.5) rectangle (7.5,5.5);
		
		\fill[line width=2.pt,fill=black, pattern=fivepointed stars,pattern color=black] (0.,4.) -- (1.,4.) -- (1.,3.) -- (2.,3.) -- (2.,1.) -- (1.,1.) -- (1.,0.) -- (0.,0.) -- cycle;
		\fill[line width=2.pt,fill=black,fill opacity=0.20000000298023224] (1.,0.) -- (1.,1.) -- (2.,1.) -- (2.,2.) -- (4.,2.) -- (4.,1.) -- (5.,1.) -- (5.,0.) -- cycle;
		\draw [line width=2.pt,color=zzttqq] (0.,0.)-- (0.,5.);
		\draw [line width=2.pt,color=zzttqq] (0.,5.)-- (7.,5.);
		\draw [line width=2.pt,color=zzttqq] (7.,5.)-- (7.,0.);
		\draw [line width=2.pt,color=zzttqq] (7.,0.)-- (0.,0.);
		\draw [line width=1.pt,dash pattern=on 1.5pt off 1.5pt] (1.,0.)-- (1.,5.);
		\draw [line width=1.pt,dash pattern=on 1.5pt off 1.5pt](2.,5.)-- (2.,0.);
		\draw [line width=1.pt,dash pattern=on 1.5pt off 1.5pt](3.,0.)-- (3.,5.);
		\draw [line width=1.pt,dash pattern=on 1.5pt off 1.5pt](4.,5.)-- (4.,0.);
		\draw [line width=1.pt,dash pattern=on 1.5pt off 1.5pt](5.,0.)-- (5.,5.);
		\draw [line width=1.pt,dash pattern=on 1.5pt off 1.5pt](6.,5.)-- (6.,0.);
		\draw[line width=1.pt,dash pattern=on 1.5pt off 1.5pt](0.,1.)-- (7.,1.);
		\draw [line width=1.pt,dash pattern=on 1.5pt off 1.5pt] (7.,2.)-- (0.,2.);
		\draw [line width=1.pt,dash pattern=on 1.5pt off 1.5pt](0.,3.)-- (7.,3.);
		\draw[line width=1.pt,dash pattern=on 1.5pt off 1.5pt] (7.,4.)-- (0.,4.);
		
		\draw [line width=1.pt] (0.,5.)-- (7.,0.);
		\draw [line width=2.pt] (0.,4.)-- (1.,4.);
		\draw [line width=2.pt] (1.,4.)-- (1.,3.);
		\draw [line width=2.pt] (1.,3.)-- (2.,3.);
		\draw [line width=2.pt] (2.,3.)-- (2.,1.);
		\draw [line width=2.pt] (2.,1.)-- (1.,1.);
		\draw [line width=2.pt] (1.,1.)-- (1.,0.);
		\draw [line width=2.pt] (1.,0.)-- (0.,0.);
		\draw [line width=2.pt] (0.,0.)-- (0.,4.);
		\draw [line width=2.pt] (1.,0.)-- (1.,1.);
		\draw [line width=2.pt] (1.,1.)-- (2.,1.);
		\draw [line width=2.pt] (2.,1.)-- (2.,2.);
		\draw [line width=2.pt] (2.,2.)-- (4.,2.);
		\draw [line width=2.pt] (4.,2.)-- (4.,1.);
		\draw [line width=2.pt] (4.,1.)-- (5.,1.);
		\draw [line width=2.pt] (5.,1.)-- (5.,0.);
		\draw [line width=2.pt] (5.,0.)-- (1.,0.);
		\end{tikzpicture}
		\caption{The sets \(\mathcal{T}_u\cup s_{\alpha}(\mathcal{T}_u)\) (starred) and \(\mathcal{T}_r\cup s_{\beta}(\mathcal{T}_r)\) (shaded).}
		\label{fig:proofwithoutgaps}
	\end{center}
	\end{figure}
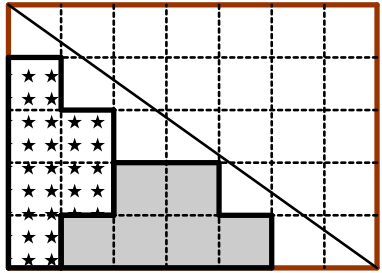

	In fact, it is easily checked that 
	\[
	B(\mathcal{T}_u)\cup B(\mathcal{T}_r)\cup s_{\beta}(B(\mathcal{T}_r))\supseteq B(LG),
	\]
	and the sets fit as shown in Figure \ref{fig:proofwithoutgaps}.
	\medskip
	
	Next we assume that \(\alpha\) is even (the case \(\beta\) even follows analogously). By Theorem \ref{thm:equifix} the set \(\mathsf{SSG}\) consists of exactly those gaps given by the lattice points \((a,\alpha/2)\) with \(1\leq a\leq \lfloor\beta/2\rfloor\). Let \(B(\mathsf{SSG})\) denote the set of border points in \(\mathsf{SSG}\), then we have a configuration as in Figure \ref{fig:42}.

	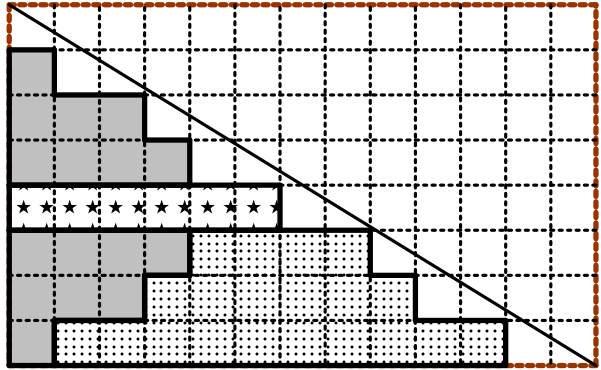
\begin{figure}[H]
		\begin{center}
		\begin{tikzpicture}[line cap=round,line join=round,>=triangle 45,x=0.6cm,y=0.6cm]
		\clip(-0.5,-0.5) rectangle (13.5,8.5);
		\fill[line width=2.8pt,fill=black,pattern=dots,pattern color=black] (1.,0.) -- (1.,1.) -- (3.,1.) -- (3.,2.) -- (4.,2.) -- (4.,3.) -- (8.,3.) -- (8.,2.) -- (9.,2.) -- (9.,1.) -- (11.,1.) -- (11.,0.) -- cycle;
		\fill[line width=3.6pt,fill=black,pattern=fivepointed stars,pattern color=black] (0.,4.) -- (0.,3.) -- (6.,3.) -- (6.,4.) -- cycle;
		\fill[line width=2.pt,fill=black,fill opacity=0.25] (0.,7.) -- (1.,7.) -- (1.,6.) -- (3.,6.) -- (3.,5.) -- (4.,5.) -- (4.,4.) -- (0.,4.) -- cycle;
		\fill[line width=2.pt,fill=black,fill opacity=0.25] (0.,3.) -- (0.,0.) -- (1.,0.) -- (1.,1.) -- (3.,1.) -- (3.,2.) -- (4.,2.) -- (4.,3.) -- cycle;
		\draw [line width=2.pt,dash pattern=on 2pt off 2pt,color=zzttqq] (0.,8.)-- (0.,0.);
		\draw [line width=2.pt,dash pattern=on 2pt off 2pt,color=zzttqq] (0.,0.)-- (13.,0.);
		\draw [line width=2.pt,dash pattern=on 2pt off 2pt,color=zzttqq] (13.,0.)-- (13.,8.);
		\draw [line width=2.pt,dash pattern=on 2pt off 2pt,color=zzttqq] (13.,8.)-- (0.,8.);
		\draw [line width=1.2pt] (0.,8.)-- (13.,0.);
		\draw [line width=1.2pt,dotted] (1.,8.)-- (1.,0.);
		\draw [line width=1.2pt,dotted] (2.,0.)-- (2.,8.);
		\draw [line width=1.2pt,dotted] (3.,0.)-- (3.,8.);
		\draw [line width=1.2pt,dotted] (4.,0.)-- (4.,8.);
		\draw [line width=1.2pt,dotted] (5.,8.)-- (5.,0.);
		\draw [line width=1.2pt,dotted] (6.,0.)-- (6.,8.);
		\draw [line width=1.2pt,dotted] (7.,8.)-- (7.,0.);
		\draw [line width=1.2pt,dotted] (8.,0.)-- (8.,8.);
		\draw [line width=1.2pt,dotted] (9.,0.)-- (9.,8.);
		\draw [line width=1.2pt,dotted] (10.,8.)-- (10.,0.);
		\draw [line width=1.2pt,dotted] (11.,0.)-- (11.,8.);
		\draw [line width=1.2pt,dotted] (12.,8.)-- (12.,0.);
		\draw [line width=1.2pt,dotted] (0.,1.)-- (13.,1.);
		\draw [line width=1.2pt,dotted] (13.,2.)-- (0.,2.);
		\draw [line width=1.2pt,dotted] (0.,3.)-- (13.,3.);
		\draw [line width=1.2pt,dotted] (13.,4.)-- (0.,4.);
		\draw [line width=1.2pt,dotted] (0.,5.)-- (13.,5.);
		\draw [line width=1.2pt,dotted] (13.,6.)-- (0.,6.);
		\draw [line width=1.2pt,dotted] (0.,7.)-- (13.,7.);
		\draw [line width=2.pt] (1.,0.)-- (1.,1.);
		\draw [line width=2.pt] (1.,1.)-- (3.,1.);
		\draw [line width=2.pt] (3.,1.)-- (3.,2.);
		\draw [line width=2.pt] (3.,2.)-- (4.,2.);
		\draw [line width=2.pt] (4.,2.)-- (4.,3.);
		\draw [line width=2.pt] (4.,3.)-- (8.,3.);
		\draw [line width=2.pt] (8.,3.)-- (8.,2.);
		\draw [line width=2.pt] (8.,2.)-- (9.,2.);
		\draw [line width=2.pt] (9.,2.)-- (9.,1.);
		\draw [line width=2.pt] (9.,1.)-- (11.,1.);
		\draw [line width=2.pt] (11.,1.)-- (11.,0.);
		\draw [line width=2.pt] (11.,0.)-- (1.,0.);
		\draw [line width=2.pt] (0.,4.)-- (0.,3.);
		\draw [line width=2.pt] (0.,3.)-- (6.,3.);
		\draw [line width=2.pt] (6.,3.)-- (6.,4.);
		\draw [line width=2.pt] (6.,4.)-- (0.,4.);
		\draw [line width=2.pt] (0.,7.)-- (1.,7.);
		\draw [line width=2.pt] (1.,7.)-- (1.,6.);
		\draw [line width=2.pt] (1.,6.)-- (3.,6.);
		\draw [line width=2.pt] (3.,6.)-- (3.,5.);
		\draw [line width=2.pt] (3.,5.)-- (4.,5.);
		\draw [line width=2.pt] (4.,5.)-- (4.,4.);
		\draw [line width=2.pt] (4.,4.)-- (0.,4.);
		\draw [line width=2.pt] (0.,4.)-- (0.,7.);
		\draw [line width=2.pt] (0.,3.)-- (0.,0.);
		\draw [line width=2.pt] (0.,0.)-- (1.,0.);
		\draw [line width=2.pt] (1.,0.)-- (1.,1.);
		\draw [line width=2.pt] (1.,1.)-- (3.,1.);
		\draw [line width=2.pt] (3.,1.)-- (3.,2.);
		\draw [line width=2.pt] (3.,2.)-- (4.,2.);
		\draw [line width=2.pt] (4.,2.)-- (4.,3.);
		\draw [line width=2.pt] (4.,3.)-- (0.,3.);
		\end{tikzpicture}
		\caption{Sets \(\mathcal{T}_u\cup s_{\alpha}(\mathcal{T}_u)\) (shaded), \(\mathsf{SSG}\) (starred), and \(\mathcal{T}_r\cup s_{\beta}(\mathcal{T}_r)\) (dotted).}
		\label{fig:42}
			\end{center}
	\end{figure}

	So it is easily seen that 
	
	\[
	B(\mathcal{T}_u)\cup B(\mathsf{SSG})\cup B(\mathcal{T}_r)\cup s_{\beta}(B(\mathcal{T}_r))\supseteq B(LG).
	\]
	
	All this together shows that the union of the triangles \(\mathcal{T}_u,\mathcal{T}_r\), their images and the set of self-symmetric gaps build a partition of the set of gaps into disjoint sets
	
	\[\mathbb{N}\setminus\Gamma=\mathcal{T}_u\bigsqcup s_{\alpha}(\mathcal{T}_u)\bigsqcup \mathsf{SSG}\bigsqcup\mathcal{T}_r\bigsqcup s_{\beta}(\mathcal{T}_r).\]
	
	We are finished as soon as the procedure to recover \(\mathbb{N}\setminus\Gamma\) ---hence \(\Gamma\)--- from the set \(\mathsf{SG}\cup \mathsf{SSG}\) will be given.
	
	\medskip
	Let us assume that \(\mathsf{SG}=\mathcal{T}_u\) (similarly for \(\mathcal{T}_r\)). Thus, we have 
	
	\[
	s_{\beta}(B(\mathcal{T}_r))=\tau(s_{\alpha}(B(\mathsf{SG}))).
	\]
	
	We distinguish two cases:
	\begin{enumerate}
		\item If \(\alpha,\beta\) are both odd, then we can recover \(s_{\beta}(\mathcal{T}_r)\) as the polyomino corresponding to the complement of \(s_\alpha(\mathsf{SG})\) in the lattice square with vertices \((0,0),(\lfloor\frac{\beta}{2}\rfloor,0),(\lfloor\frac{\beta}{2}\rfloor,\lfloor\frac{\alpha}{2}\rfloor)\), and \((0,\lfloor\frac{\alpha}{2}\rfloor)\).
		\medskip
		
		\item If \(\alpha\) or \(\beta\) is even, then consider the polyomino \(\mathsf{SSG}\cup s_\alpha(\mathsf{SG})\). Thus, we can recover  \(s_{\beta}(\mathcal{T}_r)\) as the polyomino corresponding to the complement of \(s_\alpha(\mathsf{SG})\cup\mathsf{SSG}\) in the lattice square with vertices \((0,0),(\lfloor\frac{\beta}{2}\rfloor,0),(\lfloor\frac{\beta}{2}\rfloor,\lfloor\frac{\alpha}{2}\rfloor)\), and \((0,\lfloor\frac{\alpha}{2}\rfloor)\).
	\end{enumerate}
	
	Observe that, if \(\mathsf{SG}=\mathcal{T}_r\), then the roles of \(\alpha,\beta\) in (2) and (3) need to be exchange. 
	\medskip
	
	In short, we have checked that in all cases we can obtain \(s_{\beta}(\mathcal{T}_r)\)~resp.~\(s_{\alpha}(\mathcal{T}_u)\), hence  \(\mathcal{T}_r\)~resp.~\(\mathcal{T}_u\) from certain linear transformations of the set \(\mathsf{SG}\cup\mathsf{SSG}\) in the lattice. Therefore, by the previous partition of the set of gaps we can reconstruct completely the set of gaps from the set \(\mathsf{SG}\cup\mathsf{SSG}\).
\end{proof}

The proof of Theorem \ref{thm:partition} shows in particular that \(\mathsf{SG}\) and \(\mathsf{SSG}\) are polyominoes, and that we can obtain the whole set \(LG\) making operations with them. 
These necessary operations which allow us to obtain the set $LG$ from \textsf{SG} and \textsf{SSG} will be called \emph{polyomino game}. 
We illustrate both the polyomino game and the proof of Theorem \ref{thm:partition} with an example.

\begin{ex}
	Let \(\Gamma=\langle 7,8 \rangle\). We start with \(\mathsf{SG}\) which in this case is \(\mathcal{T}_r\). Then the set of gaps represented in \(\mathcal{T}_r\) is \(\{5,6,13\}\) (see Figure \ref{fig:43}). Now, we consider the polyomino \(s_{\beta}(\mathsf{SG})\cup\mathsf{SSG}\) which represents \(\{4,12,19,20,27,34\}\) inside the rectangle of vertices \((0,0),(4,0),(4,3),(0,3)\) (see Figure \ref{fig:44}).
	\medskip
	
	 After that, we consider the complement in the rectangle of vertices \((0,0),(4,0),(4,3),(0,3)\) of the polyomino \(s_{\beta}(\mathsf{SG})\cup\mathsf{SSG}\) which represents the set of gaps \(\{41,33,26,25,18,11\}\) and we apply the map \(s_\alpha\) as we can see in Figure \ref{fig:45}. Finally, we put all polyominoes together to give rise to \(\mathbb{N}\setminus\Gamma\) as shown in Figure \ref{fig:46}.
\vspace{10cm}

	\begin{multicols}{2}
		\begin{figure}[H]
			\begin{center}
			\begin{tikzpicture}[line cap=round,line join=round,>=triangle 45,x=0.6cm,y=0.6cm]
			\clip(-0.5,-0.5) rectangle (8.5,7.5);
			
			\fill[line width=2.pt,color=zzttqq,fill=zzttqq,fill opacity=0.20] (4.,2.) -- (5.,2.) -- (5.,1.) -- (6.,1.) -- (6.,0.) -- (4.,0.) -- cycle;
			\draw [line width=1.pt,dash pattern=on 3pt off 3pt] (0.,7.)-- (0.,0.);
			\draw [line width=1.pt,dash pattern=on 3pt off 3pt]  (0.,0.)-- (8.,0.);
			\draw [line width=1.pt,dash pattern=on 3pt off 3pt]  (8.,0.)-- (8.,7.);
			\draw[line width=1.pt,dash pattern=on 3pt off 3pt]  (8.,7.)-- (0.,7.);
			\draw [line width=1.pt,dash pattern=on 3pt off 3pt]  (1.,7.)-- (1.,0.);
			\draw [line width=1.pt,dash pattern=on 3pt off 3pt]  (2.,0.)-- (2.,7.);
			\draw [line width=1.pt,dash pattern=on 3pt off 3pt]  (3.,0.)-- (3.,7.);
			\draw[line width=1.pt,dash pattern=on 3pt off 3pt]  (4.,7.)-- (4.,0.);
			\draw[line width=1.pt,dash pattern=on 3pt off 3pt]  (5.,0.)-- (5.,7.);
			\draw [line width=1.pt,dash pattern=on 3pt off 3pt]  (6.,7.)-- (6.,0.);
			\draw [line width=1.pt,dash pattern=on 3pt off 3pt]  (7.,0.)-- (7.,7.);
			\draw[line width=1.pt,dash pattern=on 3pt off 3pt]  (0.,1.)-- (8.,1.);
			\draw [line width=1.pt,dash pattern=on 3pt off 3pt]  (8.,2.)-- (0.,2.);
			\draw[line width=1.pt,dash pattern=on 3pt off 3pt]  (0.,3.)-- (8.,3.);
			\draw [line width=1.pt,dash pattern=on 3pt off 3pt] (8.,4.)-- (0.,4.);
			\draw[line width=1.pt,dash pattern=on 3pt off 3pt]  (0.,5.)-- (8.,5.);
			\draw [line width=1.pt,dash pattern=on 3pt off 3pt]  (8.,6.)-- (0.,6.);
			\draw[line width=1.5pt]  (0.,7.)-- (8.,0.);
			\draw [line width=2.pt,color=zzttqq] (4.,2.)-- (5.,2.);
			\draw [line width=1.2pt,color=zzttqq] (5.,2.)-- (5.,1.);
			\draw [line width=1.2pt,color=zzttqq] (5.,1.)-- (6.,1.);
			\draw [line width=1.2pt,color=zzttqq] (6.,1.)-- (6.,0.);
			\draw [line width=1.2pt,color=zzttqq] (6.,0.)-- (4.,0.);
			\draw [line width=1.2pt,color=zzttqq] (4.,0.)-- (4.,2.);
			
			\draw (4,0.9) node[anchor=north west] {\small 13};
			\draw (5,0.9) node[anchor=north west] {\small 6};
			\draw (4,1.9) node[anchor=north west] {\small 5};
			
			\end{tikzpicture}
			\caption{The set \(\mathsf{SG}\) (shaded).}
			\label{fig:43}
		\end{center}
		\end{figure}
		\begin{figure}[H]
			\begin{center}
			\begin{tikzpicture}[line cap=round,line join=round,>=triangle 45,x=0.6cm,y=0.6cm]
			\clip(-0.5,-0.5) rectangle (8.5,7.5);
			
			\fill[line width=2.pt,color=zzttqq,fill=zzttqq,fill opacity=0.20] (3.,2.) -- (2.,2.) -- (2.,1.) -- (1.,1.) -- (1.,0.) -- (3.,0.) -- cycle;
			\fill[line width=2.pt,color=zzttqq,fill=zzttqq,pattern=north east lines,pattern color=zzttqq] (3.,3.) -- (4.,3.) -- (4.,0.) -- (3.,0.) -- cycle;
			\draw [line width=1.pt,dash pattern=on 3pt off 3pt] (0.,7.)-- (0.,0.);
			\draw[line width=1.pt,dash pattern=on 3pt off 3pt] (0.,0.)-- (8.,0.);
			\draw [line width=1.pt,dash pattern=on 3pt off 3pt]  (8.,0.)-- (8.,7.);
			\draw[line width=1.pt,dash pattern=on 3pt off 3pt]  (8.,7.)-- (0.,7.);
			\draw [line width=1.pt,dash pattern=on 3pt off 3pt]  (1.,7.)-- (1.,0.);
			\draw [line width=1.pt,dash pattern=on 3pt off 3pt]  (2.,0.)-- (2.,7.);
			\draw [line width=1.pt,dash pattern=on 3pt off 3pt]  (3.,0.)-- (3.,7.);
			\draw[line width=1.pt,dash pattern=on 3pt off 3pt]  (4.,7.)-- (4.,0.);
			\draw[line width=1.pt,dash pattern=on 3pt off 3pt]  (5.,0.)-- (5.,7.);
			\draw [line width=1.pt,dash pattern=on 3pt off 3pt]  (6.,7.)-- (6.,0.);
			\draw [line width=1.pt,dash pattern=on 3pt off 3pt]  (7.,0.)-- (7.,7.);
			\draw[line width=1.pt,dash pattern=on 3pt off 3pt]  (0.,1.)-- (8.,1.);
			\draw [line width=1.pt,dash pattern=on 3pt off 3pt]  (8.,2.)-- (0.,2.);
			\draw[line width=1.pt,dash pattern=on 3pt off 3pt]  (0.,3.)-- (8.,3.);
			\draw [line width=1.pt,dash pattern=on 3pt off 3pt] (8.,4.)-- (0.,4.);
			\draw[line width=1.pt,dash pattern=on 3pt off 3pt]  (0.,5.)-- (8.,5.);
			\draw [line width=1.pt,dash pattern=on 3pt off 3pt]  (8.,6.)-- (0.,6.);
			\draw[line width=1.5pt]  (0.,7.)-- (8.,0.);
			\draw [line width=1.2pt,color=zzttqq] (3.,2.)-- (2.,2.);
			\draw  [line width=1.2pt,color=zzttqq] (2.,2.)-- (2.,1.);
			\draw  [line width=1.2pt,color=zzttqq] (2.,1.)-- (1.,1.);
			\draw  [line width=1.2pt,color=zzttqq] (1.,1.)-- (1.,0.);
			\draw  [line width=1.2pt,color=zzttqq] (1.,0.)-- (3.,0.);
			\draw  [line width=1.2pt,color=zzttqq] (3.,0.)-- (3.,2.);
			\draw  [line width=1.2pt,color=zzttqq] (3.,3.)-- (4.,3.);
			\draw  [line width=1.2pt,color=zzttqq] (4.,3.)-- (4.,0.);
			\draw  [line width=1.2pt,color=zzttqq] (4.,0.)-- (3.,0.);
			\draw  [line width=1.2pt,color=zzttqq] (3.,0.)-- (3.,3.);
			
			\draw (1,0.9) node[anchor=north west] {\small 34};
			\draw (2,0.9) node[anchor=north west] {\small 27};
			\draw (3,0.9) node[anchor=north west] {\small 20};
			
			\draw (2,1.9) node[anchor=north west] {\small 19};
			\draw (3,1.9) node[anchor=north west] {\small 12};
			
			\draw (3,2.9) node[anchor=north west] {\small 4};
			\end{tikzpicture}
			
			\caption{The sets \(s_{\beta}(\mathsf{SG})\) (shaded) and \(\mathsf{SSG}\) (striped).}
			\label{fig:44}
		\end{center}
		\end{figure}
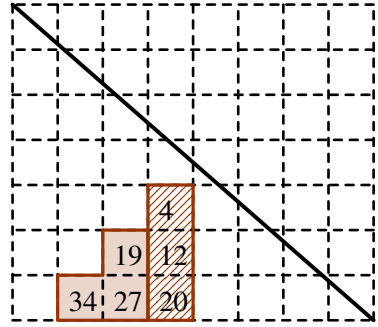

	\end{multicols}
	
	\begin{multicols}{2}
		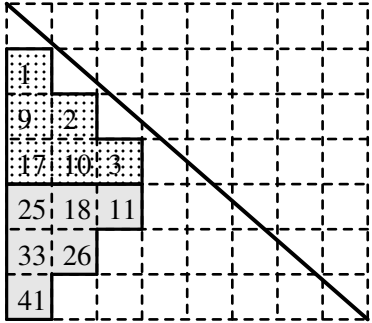
\begin{figure}[H]
			\begin{center}
				
			\begin{tikzpicture}[line cap=round,line join=round,>=triangle 45,x=0.6cm,y=0.6cm]
			\clip(-0.5,-0.5) rectangle (8.5,7.5);
			
			\fill[line width=2.pt,fill=black,fill opacity=0.10000000149011612] (0.,3.) -- (3.,3.) -- (3.,2.) -- (2.,2.) -- (2.,1.) -- (1.,1.) -- (1.,0.) -- (0.,0.) -- cycle;
			\fill[line width=2.pt,fill=black,pattern=dots,pattern color=black] (0.,3.) -- (3.,3.) -- (3.,4.) -- (2.,4.) -- (2.,5.) -- (1.,5.) -- (1.,6.) -- (0.,6.) -- cycle;
			\draw [line width=1.pt,dash pattern=on 3pt off 3pt] (0.,7.)-- (0.,0.);
			\draw [line width=1.pt,dash pattern=on 3pt off 3pt]  (0.,0.)-- (8.,0.);
			\draw [line width=1.pt,dash pattern=on 3pt off 3pt]  (8.,0.)-- (8.,7.);
			\draw[line width=1.pt,dash pattern=on 3pt off 3pt]  (8.,7.)-- (0.,7.);
			\draw [line width=1.pt,dash pattern=on 3pt off 3pt]  (1.,7.)-- (1.,0.);
			\draw [line width=1.pt,dash pattern=on 3pt off 3pt]  (2.,0.)-- (2.,7.);
			\draw [line width=1.pt,dash pattern=on 3pt off 3pt]  (3.,0.)-- (3.,7.);
			\draw[line width=1.pt,dash pattern=on 3pt off 3pt]  (4.,7.)-- (4.,0.);
			\draw[line width=1.pt,dash pattern=on 3pt off 3pt]  (5.,0.)-- (5.,7.);
			\draw [line width=1.pt,dash pattern=on 3pt off 3pt]  (6.,7.)-- (6.,0.);
			\draw [line width=1.pt,dash pattern=on 3pt off 3pt]  (7.,0.)-- (7.,7.);
			\draw[line width=1.pt,dash pattern=on 3pt off 3pt]  (0.,1.)-- (8.,1.);
			\draw [line width=1.pt,dash pattern=on 3pt off 3pt]  (8.,2.)-- (0.,2.);
			\draw[line width=1.pt,dash pattern=on 3pt off 3pt]  (0.,3.)-- (8.,3.);
			\draw [line width=1.pt,dash pattern=on 3pt off 3pt] (8.,4.)-- (0.,4.);
			\draw[line width=1.pt,dash pattern=on 3pt off 3pt]  (0.,5.)-- (8.,5.);
			\draw [line width=1.pt,dash pattern=on 3pt off 3pt]  (8.,6.)-- (0.,6.);
			\draw[line width=1.5pt]  (0.,7.)-- (8.,0.);
			\draw [line width=1.2pt] (0.,3.)-- (3.,3.);
			\draw [line width=1.2pt] (3.,3.)-- (3.,2.);
			\draw [line width=1.2pt] (3.,2.)-- (2.,2.);
			\draw [line width=1.2pt] (2.,2.)-- (2.,1.);
			\draw [line width=1.2pt] (2.,1.)-- (1.,1.);
			\draw [line width=1.2pt] (1.,1.)-- (1.,0.);
			\draw [line width=1.2pt] (1.,0.)-- (0.,0.);
			\draw [line width=1.2pt] (0.,0.)-- (0.,3.);
			\draw [line width=1.2pt] (0.,3.)-- (3.,3.);
			\draw [line width=1.2pt] (3.,3.)-- (3.,4.);
			\draw [line width=1.2pt] (3.,4.)-- (2.,4.);
			\draw [line width=1.2pt] (2.,4.)-- (2.,5.);
			\draw [line width=1.2pt] (2.,5.)-- (1.,5.);
			\draw [line width=1.2pt] (1.,5.)-- (1.,6.);
			\draw [line width=1.2pt] (1.,6.)-- (0.,6.);
			\draw [line width=1.2pt] (0.,6.)-- (0.,3.);
			
			\draw (0,0.9) node[anchor=north west] {\small 41};
			\draw (0,1.9) node[anchor=north west] {\small 33};
			\draw (0,2.9) node[anchor=north west] {\small 25};
			\draw (0,3.9) node[anchor=north west] {\small 17};
			\draw (0,4.9) node[anchor=north west] {\small 9};
			\draw (0,5.9) node[anchor=north west] {\small 1};
			
			\draw (1,1.9) node[anchor=north west] {\small 26};
			\draw (1,2.9) node[anchor=north west] {\small 18};
			\draw (1,3.9) node[anchor=north west] {\small 10};
			\draw (1,4.9) node[anchor=north west] {\small 2};
			
			\draw (2,2.9) node[anchor=north west] {\small 11};
			\draw (2,3.9) node[anchor=north west] {\small 3};
			\end{tikzpicture}
			\caption{The set \(\mathcal{T}_u\cup s_{\alpha}(\mathcal{T}_u)\).}
			\label{fig:45}
			
		\end{center}
		\end{figure}
		\begin{figure}[H]
			\begin{center}
			
			\begin{tikzpicture}[line cap=round,line join=round,>=triangle 45,x=0.6cm,y=0.6cm]
			\clip(-0.5,-0.5) rectangle (8.5,7.5);
			
			\fill[line width=2.pt,color=zzttqq,fill=zzttqq,fill opacity=0.20] (3.,2.) -- (2.,2.) -- (2.,1.) -- (1.,1.) -- (1.,0.) -- (3.,0.) -- cycle;
			\fill[line width=2.pt,color=zzttqq,fill=zzttqq,pattern=north east lines,pattern color=zzttqq] (3.,3.) -- (4.,3.) -- (4.,0.) -- (3.,0.) -- cycle;
			\fill[line width=2.pt,color=zzttqq,fill=zzttqq,fill opacity=0.20] (4.,2.) -- (5.,2.) -- (5.,1.) -- (6.,1.) -- (6.,0.) -- (4.,0.) -- cycle;
			\fill[line width=2.pt,fill=black,fill opacity=0.10000000149011612] (0.,3.) -- (3.,3.) -- (3.,2.) -- (2.,2.) -- (2.,1.) -- (1.,1.) -- (1.,0.) -- (0.,0.) -- cycle;
			\fill[line width=2.pt,fill=black,pattern=dots,pattern color=black] (0.,3.) -- (3.,3.) -- (3.,4.) -- (2.,4.) -- (2.,5.) -- (1.,5.) -- (1.,6.) -- (0.,6.) -- cycle;
			\draw [line width=1.pt,dash pattern=on 3pt off 3pt] (0.,7.)-- (0.,0.);
			\draw [line width=1.pt,dash pattern=on 3pt off 3pt]  (0.,0.)-- (8.,0.);
			\draw [line width=1.pt,dash pattern=on 3pt off 3pt]  (8.,0.)-- (8.,7.);
			\draw[line width=1.pt,dash pattern=on 3pt off 3pt]  (8.,7.)-- (0.,7.);
			\draw [line width=1.pt,dash pattern=on 3pt off 3pt]  (1.,7.)-- (1.,0.);
			\draw [line width=1.pt,dash pattern=on 3pt off 3pt]  (2.,0.)-- (2.,7.);
			\draw [line width=1.pt,dash pattern=on 3pt off 3pt]  (3.,0.)-- (3.,7.);
			\draw[line width=1.pt,dash pattern=on 3pt off 3pt]  (4.,7.)-- (4.,0.);
			\draw[line width=1.pt,dash pattern=on 3pt off 3pt]  (5.,0.)-- (5.,7.);
			\draw [line width=1.pt,dash pattern=on 3pt off 3pt]  (6.,7.)-- (6.,0.);
			\draw [line width=1.pt,dash pattern=on 3pt off 3pt]  (7.,0.)-- (7.,7.);
			\draw[line width=1.pt,dash pattern=on 3pt off 3pt]  (0.,1.)-- (8.,1.);
			\draw [line width=1.pt,dash pattern=on 3pt off 3pt]  (8.,2.)-- (0.,2.);
			\draw[line width=1.pt,dash pattern=on 3pt off 3pt]  (0.,3.)-- (8.,3.);
			\draw [line width=1.pt,dash pattern=on 3pt off 3pt] (8.,4.)-- (0.,4.);
			\draw[line width=1.pt,dash pattern=on 3pt off 3pt]  (0.,5.)-- (8.,5.);
			\draw [line width=1.pt,dash pattern=on 3pt off 3pt]  (8.,6.)-- (0.,6.);
			\draw[line width=1.5pt]  (0.,7.)-- (8.,0.);
			\draw [line width=1.2pt] (0.,3.)-- (3.,3.);
			\draw [line width=1.2pt] (3.,3.)-- (3.,2.);
			\draw [line width=1.2pt] (3.,2.)-- (2.,2.);
			\draw [line width=1.2pt] (2.,2.)-- (2.,1.);
			\draw [line width=1.2pt] (2.,1.)-- (1.,1.);
			\draw [line width=1.2pt] (1.,1.)-- (1.,0.);
			\draw [line width=1.2pt] (1.,0.)-- (0.,0.);
			\draw [line width=1.2pt] (0.,0.)-- (0.,3.);
			\draw [line width=1.2pt] (0.,3.)-- (3.,3.);
			\draw [line width=1.2pt] (3.,3.)-- (3.,4.);
			\draw [line width=1.2pt] (3.,4.)-- (2.,4.);
			\draw [line width=1.2pt] (2.,4.)-- (2.,5.);
			\draw [line width=1.2pt] (2.,5.)-- (1.,5.);
			\draw [line width=1.2pt] (1.,5.)-- (1.,6.);
			\draw [line width=1.2pt] (1.,6.)-- (0.,6.);
			\draw [line width=1.2pt] (0.,6.)-- (0.,3.);
			
			\draw (0,0.9) node[anchor=north west] {\small 41};
			\draw (0,1.9) node[anchor=north west] {\small 33};
			\draw (0,2.9) node[anchor=north west] {\small 25};
			\draw (0,3.9) node[anchor=north west] {\small 17};
			\draw (0,4.9) node[anchor=north west] {\small 9};
			\draw (0,5.9) node[anchor=north west] {\small 1};
			
			\draw (1,1.9) node[anchor=north west] {\small 26};
			\draw (1,2.9) node[anchor=north west] {\small 18};
			\draw (1,3.9) node[anchor=north west] {\small 10};
			\draw (1,4.9) node[anchor=north west] {\small 2};
			
			\draw (2,2.9) node[anchor=north west] {\small 11};
			\draw (2,3.9) node[anchor=north west] {\small 3};
			
			\draw (4,0.9) node[anchor=north west] {\small 13};
			\draw (5,0.9) node[anchor=north west] {\small 6};
			\draw (4,1.9) node[anchor=north west] {\small 5};
			
			\draw (1,0.9) node[anchor=north west] {\small 34};
			\draw (2,0.9) node[anchor=north west] {\small 27};
			\draw (3,0.9) node[anchor=north west] {\small 20};
			
			\draw (2,1.9) node[anchor=north west] {\small 19};
			\draw (3,1.9) node[anchor=north west] {\small 12};
			
			\draw (3,2.95) node[anchor=north west] {\small 4};
			\end{tikzpicture}
			\caption{Lattice representation of the gap set \(\mathbb{N}\setminus\Gamma\).}
			\label{fig:46}
			
		\end{center}
		\end{figure}
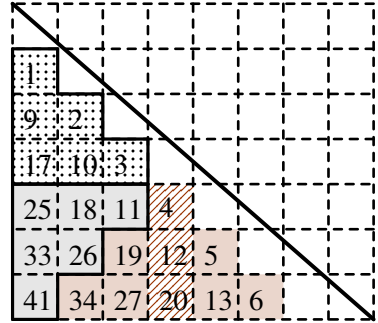
	\end{multicols}
\end{ex}

Since the polyominoes are lattice polygons, we are able to present formulas for the cardinal of the sets of supersymmetric and self-symmetric gaps.
\begin{proposition}\label{prop: formula cardinalsym}
	Let \(\Gamma=\langle\alpha,\beta\rangle\) be a numerical semigroup. Then
	
	$$
	|\mathsf{SSG}|=\Bigg\{\begin{array}{cc}
	0&\text{if \(\alpha,\beta\) are odd}\\
	(\beta-1)/2 &\text{if \(\alpha\) even}\\
	(\alpha-1)/2 &\text{if \(\beta\) even}
	\end{array}
	$$
	Moreover,
		
	$$		
	|\mathsf{SG}|=\left \{\begin{array}{ll}
	\displaystyle \sum_{j=1}^{\lfloor\frac{\alpha}{2}\rfloor-1}  \Big \lfloor \frac{j\beta}{\alpha} \Big \rfloor&\mbox{if}\ \;\mathsf{SG}=\mathcal{T}_u\\
	&\\
	\displaystyle \sum_{j=h}^{\alpha-1} \Big (\Big \lfloor \frac{j\beta}{\alpha} \Big \rfloor -\Big \lfloor \frac{\beta}{2} \Big \rfloor   \Big ) &\mbox{if}\ \;\mathsf{SG}=\mathcal{T}_r
	\end{array}
	\right.
	$$
	where $\displaystyle h=  \frac{\alpha}{2} +1$ if $\alpha$ is even, and  $\displaystyle h=\Big \lfloor \frac{\alpha}{2}\Big \rfloor $ if $\alpha$ is odd.
\end{proposition}
\begin{proof}
	The formula for \(|\mathsf{SSG}|\) is a direct consequence of Theorem \ref{thm:equifix}. So let us prove the formula for \(|\mathsf{SG}|\). We start by showing that \((a,b)\in RB(\mathcal{T}_u)\)~resp.~\((a,b)\in RB(\mathcal{T}_r)\) if it is of the form \((\lfloor j\beta /\alpha\rfloor,\alpha-j)\) with \(j=1,\dots,\lfloor \alpha/2\rfloor-1\),~resp.~ \(j=h,\dots,\alpha-1\), where $h= \alpha/2  +1$ if $\alpha$ is even, and  $h=\lfloor \alpha/2 \rfloor$ if $\alpha$ is odd. Obviously, \((\lfloor j\beta /\alpha\rfloor,\alpha-j)\) lies always on the right-hand sided border, since 
	
	\[\alpha\beta -\lfloor j\beta /\alpha\rfloor\alpha-(\alpha-j)\beta \geq 0\quad\text{and}\quad \alpha\beta -(\lfloor j\beta /\alpha\rfloor+1)\alpha-(\alpha-j)\beta \leq 0.\]
	
	Now, observe that by definition the points on \(RB(\mathcal{T}_u)\) have second coordinate varying from \(\alpha-1\) to \(\alpha-\lfloor\alpha/2\rfloor +1\). For the points in  \(RB(\mathcal{T}_r)\) we need to distinguish two cases: if \(\alpha\) is even, then the points with second coordinate \(\alpha-\alpha/2\) are self-symmetric gaps so they do not belong to \(\mathcal{T}_r\) and we need to start the summation running from \(\alpha/2 +1\) on. If \(\alpha\) is odd, then there are no self-symmetric gaps of the previous form. The unique self-symmetric gaps may be those with coordinates \((\beta/2,\lfloor \alpha/2 \rfloor)\), but if one of them is actually a border point, then it adds zero in the summation.
\end{proof}

\subsection{Fundamental gaps vs supersymmetric gaps and self-symmetric gaps}\label{subsec:fundvssym}

The fundamental gaps for semigroups of the form \(\Gamma=\langle\alpha,\beta \rangle\) are explicitly described by Rosales in \cite[Theorem 9]{Rosales}. As part of the proof, he characterized the elements \(x\in\mathbb{N}\setminus\Gamma\) such that \(2x\in\Gamma\). From this characterization we are able to prove the following.
\begin{proposition}\label{prop:red}
	Let \(\Gamma=\langle\alpha,\beta \rangle\), and let \(x\in\mathbb{N}\setminus\Gamma\) be a gap of $\Gamma$. Then the following are equivalent:
	\begin{enumerate}
		\item \(2x\in\Gamma\); 
		\item \(x=\alpha\beta-a\alpha-b\beta\) with \(1\leq a\leq\beta/2\) and \(1\leq b\leq \alpha/2\);
		\item \(W(x)\leq 0\).
	\end{enumerate}
\end{proposition}  
\begin{proof}
	The equivalence \((1)\Leftrightarrow (2)\) is \cite[Proposition~4]{Rosales}, and \((2)\Leftrightarrow (3)\) is a straightforward computation from the formula given in Proposition \ref{prop:wilfgap}.
\end{proof}

In particular, non-positive Wilf number is a necessary condition for a gap to be a fundamental gap. 
\begin{corollary}
	Let \(\Gamma=\langle\alpha,\beta \rangle\), and let \(x\in\mathbb{N}\setminus\Gamma\) be a gap of $\Gamma$. If \(x\in\mathcal{FG}(\Gamma)\), then \(W(x)\leq 0\).
\end{corollary}
Notice that the converse is not true: consider the semigroup \(\Gamma=\langle 8, 13 \rangle\) and take the gap \(25\), then \(W(25)=-9<0\) but \(25\notin\mathcal{FG}(\Gamma)\).

We recall that a subset $X$ of  the set of nonnegative integers \(H\)--determines a numerical semigroup \(\Gamma\) if \(\Gamma\) is the maximal numerical semigroup with respect to set inclusion such that \(X\subset\mathbb{N}\setminus\Gamma\). 
Under this description of \(\Gamma\), the set of fundamental gaps is the smallest subset \(H\)--determining \(\Gamma\). Moreover, Rosales et al. \cite{Rosetal} proved the following important result about minimality of the fundamental gaps with respect the \(H\)--determinacy.
\begin{proposition}\cite[Corollary~7]{Rosetal}
	Let \(\Gamma\) be a numerical semigroup and let \(X\subset\mathbb{N}\setminus\Gamma\). The set \(X\) \(H\)--determines \(\Gamma\) if and only if \(\mathcal{FG}(\Gamma)\subset X\).
\end{proposition}

On the other hand, we have proven in Theorem \ref{thm:partition} that \(\mathsf{SG}\cup\mathsf{SSG}\) completely determines \(\Gamma\). In this way, it is natural to compare \(\mathsf{SG}\cup\mathsf{SSG}\) with \(\mathcal{FG}(\Gamma)\). However, this comparison is not set-theoretically possible since we do not have inclusion relations, i.e. \(\mathsf{SG}\cap\mathcal{FG}(\Gamma)=\emptyset\) and \(\mathsf{SSG}\subset\{x\in\mathbb{N}\setminus\Gamma:\;2x\in\Gamma\}\) but in general \(\mathsf{SSG}\nsubseteqq\mathcal{FG}(\Gamma)\) as Example \ref{ex:grande} shows; this example also shows that if \(\mathsf{SG}=\mathcal{T}_u\)~resp.~\(\mathcal{T}_r\) then \(s_{\alpha}(\mathsf{SG})\)~resp.~\(s_{\beta}(\mathsf{SG})\) does not need to be contained in \(\mathcal{FG}(\Gamma)\).

This means that, in general, the set \(\mathsf{SG}\cup\mathsf{SSG}\) does not \(H\)--determines \(\Gamma\), but it determines \(\Gamma\) in the sense that \(\Gamma\) can be recovered from \(\mathsf{SG}\cup\mathsf{SSG}\). Moreover, the polyomino game cannot be recovered from the set of fundamental gaps. 
Then, the most we can do is to compare the cardinality of both sets:
\begin{proposition}\label{prop:uff}
	Let \(\Gamma=\langle\alpha,\beta \rangle\) be a numerical semigroup with $\alpha >2$, then
	\[
	|\mathsf{SG}\cup\mathsf{SSG}|\leq|\mathcal{FG}(\Gamma)|.
	\]
\end{proposition}
\begin{proof}
	We will distinguish three cases depending on the parity of the semigroup generators.
\medskip
		
	\noindent Case (A): If $\alpha$ and $\beta$ are both of them odd numbers, then $\mathsf{SSG}=\emptyset$ and
	
	\begin{equation}\label{eqn:auxtriangle}
	\min \{|\tau_u|,|\tau_r|\} \leq \frac{1}{8}(\alpha -1)(\beta -1).
	\end{equation}
	
	In order to prove Equation \eqref{eqn:auxtriangle} we observe that 
	
	$$
	|\mathbb{N}\setminus \Gamma| = \sum_{j=1}^{\alpha-1}\Big \lfloor \frac {j\beta}{\alpha} \Big \rfloor =\frac{1}{2}(\alpha -1)(\beta -1) = 2 \Big \lfloor \frac {\alpha}{2} \Big \rfloor  \Big \lfloor \frac {\beta}{2} \Big \rfloor = 2 \cdot |\{x \in \mathbb{N}\setminus \Gamma : 2x\in \Gamma\}|.
	$$
	Therefore,  $|\tau_u|+|\tau_r|= |\mathbb{N}\setminus \Gamma| -\frac{1}{4}(\alpha -1)(\beta -1)=\frac{1}{4}(\alpha -1)(\beta -1)$. This give us directly Equation \eqref{eqn:auxtriangle}. Now, in view of \cite[Corollary~11]{Rosales} it is enough to show that
	
	$$
	\frac{1}{4}(\alpha -1)(\beta -1) - \Big \lceil \frac{\alpha -3}{6}  \Big \rceil \Big \lceil \frac{\beta -3}{6}  \Big \rceil \geq \frac{1}{8}(\alpha -1)(\beta -1),
	$$
	which is equivalent to the inequality
	
	$$
	\frac{1}{8}(\alpha -1)(\beta -1) \geq  \Big \lceil \frac{\alpha -3}{6}  \Big \rceil \Big \lceil \frac{\beta -3}{6}  \Big \rceil .
	$$
	This is true: since 
	
	\begin{align*}
	\Big \lceil \frac{\alpha -3}{6}  \Big \rceil \Big \lceil \frac{\beta -3}{6}  \Big \rceil &= \Big (\Big \lfloor \frac{\alpha -3}{6}  \Big \rfloor+1 \Big ) \Big (\Big \lfloor \frac{\beta -3}{6}  \Big \rfloor + 1\Big )\leq
	\frac{1}{36} (\alpha \beta +3\alpha + 3\beta+9),
	\end{align*}
	we just need to realize that
	
	$$
	\frac{1}{36} (\alpha \beta +3\alpha + 3\beta+9) \leq \frac{1}{8}(\alpha \beta - \alpha -\beta +1),
	$$
	which leads to the inequality
	
	$$
	7\alpha\beta-15\alpha-15\beta-9\geq 0.
	$$
	This holds for $\alpha =3$ and $\beta \geq 11$ odd as well as for any $\alpha \geq 5 $ odd and $\beta>\alpha$ odd. The cases $\alpha=3, \beta =5$ and $\alpha=3, \beta =7$ must be treated separately, see Figure \ref{fig:test}. In the first case, an easy computation shows that $|\mathcal{FG}(\Gamma)|=2$ and $|\tau_u|=1$, $|\tau_r|=1$, and the result follows; in the second case, we have that $|\mathcal{FG}(\Gamma)|=3$, and $|\tau_u|=2$, $|\tau_r|=1$, so the result remains also true.

	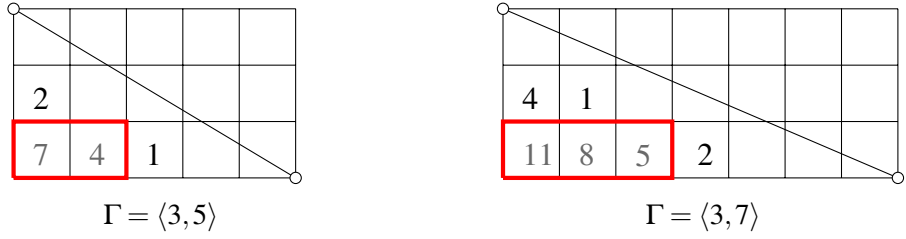
\begin{figure}
		\begin{center}

		\begin{subfigure}{.4\textwidth}
			\begin{tikzpicture}[scale=0.75]
			\draw[] (0,0) grid [step=1cm](5,3);
			\draw[] (0,3) -- (5,0);
			\draw[ultra thick,red] (0,0) -- (2,0) -- (2,1) -- (0,1) -- (0,0);
			\draw[fill=white] (0,3) circle [radius=0.1]; 
			\draw[fill=white] (5,0) circle [radius=0.1]; 
			
			\node [below right][DimGray] at (0.15,0.8) {$7$};
			\node [below right][DimGray] at (1.15,0.8) {$4$};
			\node [below right] at (2.15,0.8) {$1$};
			
			\node [below right] at (0.15,1.8) {$2$};
			
			
			\end{tikzpicture}
			\caption*{$\Gamma=\langle 3,5\rangle$ \ \ \ \ \ \ \ \ \ \ \ \ \ \ \ \ \ \ \ \ \ \ \ \ } \label{fig:sub1}
		\end{subfigure}	
		\begin{subfigure}{.4\textwidth}
			\begin{tikzpicture}[scale=0.75]
			\draw[] (0,0) grid [step=1cm](7,3);
			\draw[] (0,3) -- (7,0);
			\draw[ultra thick,red] (0,0) -- (3,0) -- (3,1) -- (0,1) --(0,0);
			\draw[fill=white] (0,3) circle [radius=0.1]; 
			
			\draw[fill=white] (7,0) circle [radius=0.1]; 
			
			\node [below right][DimGray] at (0.15,0.8) {$11$};
			\node [below right][DimGray] at (1.15,0.8) {$8$};
			\node [below right][DimGray] at (2.15,0.8) {$5$};
			\node [below right] at (3.25,0.8) {$2$};
			
			\node [below right] at (0.15,1.8) {$4$};
			\node [below right] at (1.15,1.8) {$1$};
			
			
			\end{tikzpicture}
			\caption*{\hspace{-0.9cm}$\Gamma=\langle 3,7\rangle$}  \label{fig:sub2}
		\end{subfigure}
		\caption{Special cases in (A) of Proposition \ref{prop:uff}.}
		\label{fig:test}
					
	\end{center}
	\end{figure}
	\vspace{0.5cm}
	
	\noindent Case (B): If $\alpha >2$ is even and $\beta>\alpha$ is odd, then $|\mathsf{SSG}|=\Big \lfloor \frac{\beta}{2}  \Big \rfloor$ and also
	
	$$
	|\tau_u|+|\tau_r|=\frac{1}{4}(\alpha-1)(\beta -1).
	$$ 
	
	By reasoning as in Case (A), it suffices to prove that
	
	$$
	\frac{1}{4}\alpha(\beta -1) - \Big \lceil \frac{\alpha -3}{6}  \Big \rceil \Big \lceil \frac{\beta -3}{6}  \Big \rceil \geq \frac{1}{8}(\alpha -1)(\beta -1) + \frac{1}{2} \Big \lfloor \frac{\beta}{2}  \Big \rfloor,
	$$
	again by \cite[Corollary~11]{Rosales}. But this leads us to Case (A) since 
	
	\[\frac{1}{4}\alpha(\beta -1)-\frac{1}{8}(\alpha -1)(\beta -1)=\frac{1}{8}(\alpha -1)(\beta -1)+\frac{\beta-1}{4}\quad\text{and}\quad \frac{\beta-1}{4}=\frac{1}{2} \Big \lfloor \frac{\beta}{2}  \Big \rfloor.\]
	
	\noindent Case (C): If $\alpha\geq 3$ is odd and $\beta>\alpha$ is even, we may repeat \emph{mutatis mutandis} the argument in Case (B), and the result follows.
	
\end{proof}

\begin{rem}
	Observe that for $\alpha =2, \beta=3$ the statement of Proposition \ref{prop:uff} holds; but this is no longer true for $\alpha=2$ and any $\beta\geq 3$ odd, since in that case $|\tau_u|=|\tau_r|=0$ and $|\mathcal{FG}(\Gamma)|=\frac{\beta-1}{2}-\big \lceil \frac{\beta -3}{6}  \big \rceil >0$.
\end{rem}

Finally, let us show in an example how all the important sets presented in this section look like.

\begin{ex}\label{ex:grande}
	Consider the numerical semigroup \(\Gamma=\langle 8,13\rangle\). In this case \(\mathsf{SG}=\mathcal{T}_u\) as we can see in Figure \ref{fig:gran}. This figure is labelled in the following manner: as usual, every lattice cell represents the gap of $\Gamma$ given by $(a,b)$, where these are the coordinates of the upper-right corner of the cell. Every cell is endowed with two numbers: the one lying on the bottom of the cell is just the corresponding gap, while the number on the top of the cell is the Wilf number of the gap.

	The figure also presents a filling code: we have shadowed the set of fundamental gaps and dotted the set of self-symmetric gaps. This makes it clear that self-symmetric gaps are not fully contained in the set of fundamental gaps neither the images by \(s_\alpha,s_\beta\) of the triangles \(\mathcal{T}_u,\mathcal{T}_r.\) The red rectangle contains the gaps $x$ such that $2x\in \Gamma$, cf. Proposition \ref{prop:red}. The polyominoes corresponding to $\mathcal{T}_r\cup s_{\alpha}(\mathcal{T}_r)$ and $\mathcal{T}_u\cup s_{\beta}(\mathcal{T}_u)$ are also distinguishable.
	
	\begin{center}
		\begin{figure}[H]\label{fig:813}
			\definecolor{ffqqtt}{rgb}{1.,0.,0.2}
			\definecolor{uququq}{rgb}{0.25098039215686274,0.25098039215686274,0.25098039215686274}
			\definecolor{zzttqq}{rgb}{0.6,0.2,0.}
			\begin{tikzpicture}[line cap=round,line join=round,>=triangle 45,x=1.0cm,y=1.0cm]
			\clip(-0.5,-0.5) rectangle (13.5,8.5);
			\draw[line width=1.pt,dashed] (0.,8.) -- (0.,0.) -- (13.,0.) -- (13.,8.) -- cycle;
			
			\fill[line width=0.pt,color=uququq,fill=uququq,fill opacity=0.20000000298023224] (0.,4.) -- (4.,4.) -- (4.,2.) -- (6.,2.) -- (6.,0.) -- (0.,0.) -- cycle;
			\draw[line width=2.pt,color=ffqqtt] (0.,4.) -- (6.,4.) -- (6.,0.) -- (0.,0.) -- cycle;
			\draw[line width=1.pt] (0.,7.) -- (1.,7.) -- (1.,6.) -- (3.,6.) -- (3.,5.) -- (4.,5.) -- (4.,4.) -- (4.,2.) -- (3.,2.) -- (3.,1.) -- (1.,1.) -- (1.,0.) -- (0.,0.) -- cycle;
			\draw[line width=1.pt] (1.,0.) -- (1.,1.) -- (3.,1.) -- (3.,2.) -- (4.,2.) -- (4.,3.) -- (8.,3.) -- (8.,2.) -- (9.,2.) -- (9.,1.) -- (11.,1.) -- (11.,0.) -- cycle;
			\fill[line width=0.pt,fill=zzttqq,pattern=north east lines, pattern color=zzttqq, fill opacity=0.5] (0.,4.) -- (0.,3.) -- (6.,3.) -- (6.,4.) -- cycle;
			
			\draw [line width=1.3pt] (0.,8.)-- (13.,0.);
			\draw [line width=1.2pt,dotted] (1.,8.)-- (1.,0.);
			\draw [line width=1.2pt,dotted] (2.,0.)-- (2.,8.);
			\draw [line width=1.2pt,dotted] (3.,0.)-- (3.,8.);
			\draw [line width=1.2pt,dotted] (4.,0.)-- (4.,8.);
			\draw [line width=1.2pt,dotted] (5.,8.)-- (5.,0.);
			\draw [line width=1.2pt,dotted] (6.,0.)-- (6.,8.);
			\draw [line width=1.2pt,dotted] (7.,8.)-- (7.,0.);
			\draw [line width=1.2pt,dotted] (8.,0.)-- (8.,8.);
			\draw [line width=1.2pt,dotted] (9.,0.)-- (9.,8.);
			\draw [line width=1.2pt,dotted] (10.,8.)-- (10.,0.);
			\draw [line width=1.2pt,dotted] (11.,0.)-- (11.,8.);
			\draw [line width=1.2pt,dotted] (12.,8.)-- (12.,0.);
			\draw [line width=1.2pt,dotted] (0.,1.)-- (13.,1.);
			\draw [line width=1.2pt,dotted] (13.,2.)-- (0.,2.);
			\draw [line width=1.2pt,dotted] (0.,3.)-- (13.,3.);
			\draw [line width=1.2pt,dotted] (13.,4.)-- (0.,4.);
			\draw [line width=1.2pt,dotted] (0.,5.)-- (13.,5.);
			\draw [line width=1.2pt,dotted] (13.,6.)-- (0.,6.);
			\draw [line width=1.2pt,dotted] (0.,7.)-- (13.,7.);
			
			\draw [line width=2.pt] (0.,7.)-- (1.,7.);
			\draw [line width=2.pt] (1.,7.)-- (1.,6.);
			\draw [line width=2.pt] (1.,6.)-- (3.,6.);
			\draw [line width=2.pt] (3.,6.)-- (3.,5.);
			\draw [line width=2.pt] (3.,5.)-- (4.,5.);
			\draw [line width=2.pt] (4.,5.)-- (4.,4.);
			\draw [line width=2.pt] (4.,4.)-- (4.,2.);
			\draw [line width=2.pt] (4.,2.)-- (3.,2.);
			\draw [line width=2.pt] (3.,2.)-- (3.,1.);
			\draw [line width=2.pt] (3.,1.)-- (1.,1.);
			\draw [line width=2.pt] (1.,1.)-- (1.,0.);
			\draw [line width=2.pt] (1.,0.)-- (0.,0.);
			\draw [line width=2.pt] (0.,0.)-- (0.,7.);
			\draw [line width=2.pt] (1.,0.)-- (1.,1.);
			\draw [line width=2.pt] (1.,1.)-- (3.,1.);
			\draw [line width=2.pt] (3.,1.)-- (3.,2.);
			\draw [line width=2.pt] (3.,2.)-- (4.,2.);
			\draw [line width=2.pt] (4.,2.)-- (4.,3.);
			\draw [line width=2.pt] (4.,3.)-- (8.,3.);
			\draw [line width=2.pt] (8.,3.)-- (8.,2.);
			\draw [line width=2.pt] (8.,2.)-- (9.,2.);
			\draw [line width=2.pt] (9.,2.)-- (9.,1.);
			\draw [line width=2.pt] (9.,1.)-- (11.,1.);
			\draw [line width=2.pt] (11.,1.)-- (11.,0.);
			\draw [line width=2.pt] (11.,0.)-- (1.,0.);

			\draw (0,0.52) node[anchor=north west] {83};
			\draw (0.45,1) node[anchor=north west] {-6};
			\draw (1,0.52) node[anchor=north west] {75};
			\draw (1.45,1) node[anchor=north west] {-9};
			\draw (2,0.52) node[anchor=north west] {67};
			\draw (2.45,1) node[anchor=north west] {-7};
			\draw (3,0.52) node[anchor=north west] {59};
			\draw (3.45,1) node[anchor=north west] {-5};
			\draw (4,0.52) node[anchor=north west] {51};
			\draw (4.45,1) node[anchor=north west] {-3};
			\draw (5,0.52) node[anchor=north west] {43};
			\draw (5.45,1) node[anchor=north west] {-1};
			\draw (6,0.52) node[anchor=north west] {35};
			\draw (6.45,1) node[anchor=north west] {1};
			
			\draw (7,0.52) node[anchor=north west] {27};
			\draw (7.45,1) node[anchor=north west] {3};
			\draw (8,0.52) node[anchor=north west] {19};
			\draw (8.45,1) node[anchor=north west] {5};
			\draw (9,0.52) node[anchor=north west] {11};
			\draw (9.45,1) node[anchor=north west] {7};
			\draw (10,0.52) node[anchor=north west] {3};
			\draw (10.45,1) node[anchor=north west] {9};

			\draw (0,1.52) node[anchor=north west] {70};
			\draw (0.45,2) node[anchor=north west] {-4};
			\draw (1,1.52) node[anchor=north west] {62};
			\draw (1.45,2) node[anchor=north west] {-8};
			\draw (2,1.52) node[anchor=north west] {54};
			\draw (2.24,2) node[anchor=north west] {-12};
			\draw (3,1.52) node[anchor=north west] {46};
			\draw (3.24,2) node[anchor=north west] {-10};
			\draw (4,1.52) node[anchor=north west] {38};
			\draw (4.45,2) node[anchor=north west] {-6};
			\draw (5,1.52) node[anchor=north west] {30};
			\draw (5.45,2) node[anchor=north west] {-2};
			\draw (6,1.52) node[anchor=north west] {22};
			\draw (6.45,2) node[anchor=north west] {2};
			\draw (7,1.52) node[anchor=north west] {14};
			\draw (7.45,2) node[anchor=north west] {6};
			\draw (8,1.52) node[anchor=north west] {6};
			\draw (8.25,2) node[anchor=north west] {10};

			\draw (0,2.52) node[anchor=north west] {57};
			\draw (0.45,3) node[anchor=north west] {-2};
			\draw (1,2.52) node[anchor=north west] {49};
			\draw (1.45,3) node[anchor=north west] {-4};
			\draw (2,2.52) node[anchor=north west] {41};
			\draw (2.45,3) node[anchor=north west] {-6};
			\draw (3,2.52) node[anchor=north west] {33};
			\draw (3.40,3) node[anchor=north west] {-8};
			\draw (4,2.52) node[anchor=north west] {25};
			\draw (4.40,3) node[anchor=north west] {-9};
			\draw (5,2.52) node[anchor=north west] {17};
			\draw (5.45,3) node[anchor=north west] {-3};
			\draw (6,2.52) node[anchor=north west] {9};
			\draw (6.45,3) node[anchor=north west] {3};
			\draw (7.,2.52) node[anchor=north west] {1};
			\draw (7.45,3) node[anchor=north west] {9};

			\draw (0,3.52) node[anchor=north west] {44};
			\draw (0.55,4) node[anchor=north west] {0};
			\draw (1,3.52) node[anchor=north west] {36};
			\draw (1.55,4) node[anchor=north west] {0};
			\draw (2,3.52) node[anchor=north west] {28};
			\draw (2.55,4) node[anchor=north west] {0};
			\draw (3,3.52) node[anchor=north west] {20};
			\draw (3.55,4) node[anchor=north west] {0};
			\draw (4,3.52) node[anchor=north west] {12};
			\draw (4.55,4) node[anchor=north west] {0};
			\draw (5,3.52) node[anchor=north west] {4};
			\draw (5.55,4) node[anchor=north west] {0};

			\draw (0,4.52) node[anchor=north west] {31};
			\draw (0.45,5) node[anchor=north west] {2};
			\draw (1,4.52) node[anchor=north west] {23};
			\draw (1.45,5) node[anchor=north west] {4};
			\draw (2,4.52) node[anchor=north west] {15};
			\draw (2.45,5) node[anchor=north west] {6};
			\draw (3,4.52) node[anchor=north west] {7};
			\draw (3.45,5) node[anchor=north west] {8};
			
			\draw (0,5.52) node[anchor=north west] {18};
			\draw (0.45,6) node[anchor=north west] {4};
			\draw (1,5.52) node[anchor=north west] {10};
			\draw (1.45,6) node[anchor=north west] {8};
			\draw (2,5.52) node[anchor=north west] {2};
			\draw (2.25,6) node[anchor=north west] {12};
			
			\draw (0,6.52) node[anchor=north west] {5};
			\draw (0.45,7) node[anchor=north west] {6};
			
			\end{tikzpicture}
			
			\caption{Polyomino game for the semigroup $\langle 8,13 \rangle$.}
			\label{fig:gran}
		\end{figure}
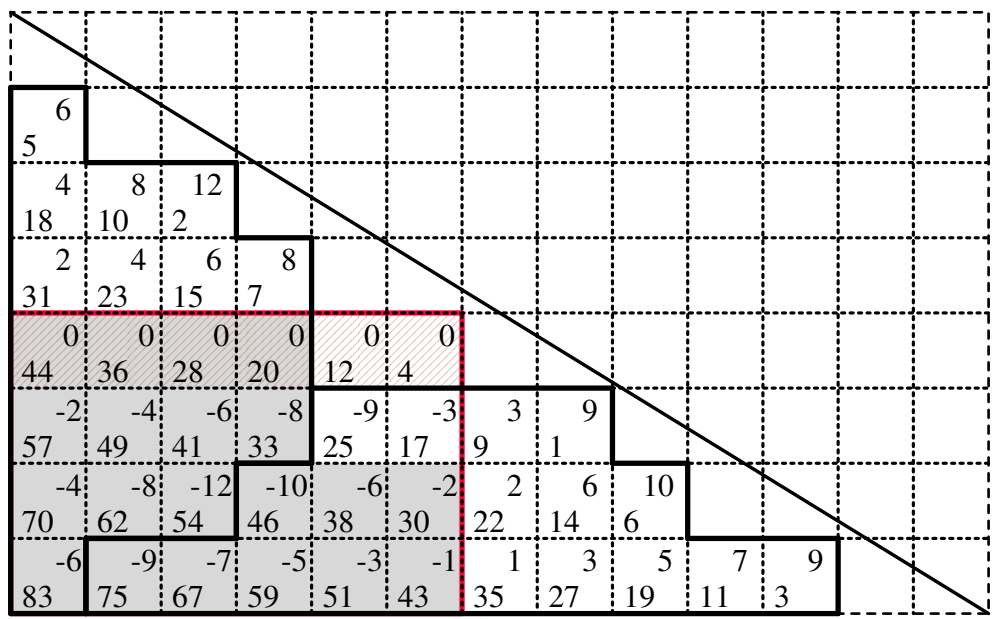
	\end{center}
\end{ex}

\subsection{Remarks on the concepts of supersymmetric and self-symmetric gaps}\label{remarksconcepts} We would like to finish the paper with a brief discussion about the possible extension of the concepts of supersymmetric and self-symmetric gaps to the case of a numerical semigroup \(\Gamma\) with arbitrary embedding dimension. 
\medskip


We remark first that our definition of supersymmetry does not coincide with the one given by Fr\"oberg, Gottlieb and H\"aggkvist in \cite{fgh}. Their notion lies on a lattice representation of each of the Ap\'ery sets with respect to all minimal generators of the semigroup in such a way that supersymmetry in the sense of \cite[Lemma 15]{fgh} means symmetry plus uniqueness of the concerned lattice representation; recall that the Ap\'ery set of $\Gamma$ with respect to a nonzero element $s\in \Gamma$ is defined to be $\{w\in \Gamma : w-s \notin \Gamma \}$. On the other hand, our notion is defined from the lattice representation of the set of gaps of the semigroup together with its properties with respect to the Wilf numbers.
\medskip

The extension of the notions of self-symmetric and supersymmetric gaps to higher embedding dimensions is trickier. The following example shows us that we cannot expect a definition through only the sign of the Wilf number of the concerned gap, since for this example all gaps have negative Wilf number; this means, the sign here is not the important issue, but the absolute value.
\begin{ex}
Let \(\Gamma=\langle 4,6,13 \rangle\). This is a symmetric semigroup which is the value semigroup of the complex plane curve singularity given by the equation $f(x,y)=x^{13}+4x^8y-x^6+2x^3y^2-y^4$; see \cite{KH} for the definition of value semigroup of a curve. The set of gaps of $\Gamma$ is $G=\{1,2,3,5,7,9,11,15\}$. We list the minimal system of generators for those $\Gamma$-semimodules $\Delta$ with embedding dimension $2$, as well as the corresponding minimal generating system for $\Delta_J=\mathrm{Syz}(\Delta)$ and its normalized $\Delta_J^{\circ}$, together with $W(\Delta_I)$:

\begin{table}[h!]
  \begin{center}
    \begin{tabular}{cccc} 
      $\Delta_I$ & $\Delta_J=\mathrm{Syz}(\Delta_I)$ & $\Delta_I^{\circ}$ & $W(\Delta_I)$   \\
      \hline\hline
      $[0,1]$ & $[13,14]$ & $[0,1]$ & $0$ \\
    $[0,2]$ & $[6,8]$ & $[0,2]$ & $0$ \\
      $[0,3]$ & $[13,16]$ & $[0,3]$ & $0$ \\
    $[0,5]$ & $[13,18]$ & $[0,5]$ & $0$ \\
      $[0,7]$ & $[13,20]$ & $[0,7]$ & $0$ \\
      $[0,9]$ & $[13,22]$ & $[0,9]$ & $0$ \\
      $[0,11]$ & $[17,19,24]$ & $[0,2,7]$ & $-2$ \\
      $[0,15]$ & $[19,21,28]$ & $$[0,2,9]$$ & $-2$ \\
   \end{tabular}
\end{center}
\end{table}

\end{ex}

In addition, the example \(\Gamma=\langle 10,14, 27\rangle\) after Theorem \ref{thm:equifix} shows that if we want to extend the concept of self-symmetric gap we cannot only focus on Wilf number zero; it seems that the notion of supersymmetry is deeper.
\medskip

Moreover, observe that the symmetries under consideration imply the following property:

\begin{proposition}\label{prop:last}
	Let \(\Gamma=\langle\alpha,\beta\rangle\) be a numerical semigroup. With the previous notation,
	\begin{itemize}
		\item[(1)] If \(\alpha\beta-a\alpha-b\beta=g\in\mathcal{T}_u\), then \(c(\Delta_{[0,g]})=c(\Delta_{[0,s_\alpha(g)]})=c(\Gamma)-a\alpha\).
		\item[(2)] If \(\alpha\beta-a\alpha-b\beta=g\in\mathcal{T}_r\), then \(c(\Delta_{[0,g]})=c(\Delta_{[0,s_\beta(g)]})=c(\Gamma)-b\beta\).
	\end{itemize}
\end{proposition}

\begin{proof}
We will prove only (1), and (2) follows \emph{mutatis mutandis}.
\medskip
	
Consider the gap \(\alpha\beta-a\alpha-b\beta=g\in\mathcal{T}_u\). Let \([h_0,h_1]\) be the minimal set of generators of \(\mathrm{Syz}(\Delta_{[0,g]})\). Then by Lemma \ref{lem:minsyz}
 $$
 \mathrm{min}\{h_0,h_1\}=\alpha\beta-b\beta,
 $$
 
and Theorem \ref{formula} implies that \(c(\Delta_{[0,g]})=c(\Gamma)-a\alpha\).
\medskip
 
Consider now \(s_\alpha(g)=\alpha\beta-a\alpha-(\alpha-b)\beta\), and denote by 
$$
[h'_0=\alpha\beta-a\alpha,h'_1=\alpha\beta-(\alpha-b)\beta]
$$
the minimal set of generators of \(\mathrm{Syz}(\Delta_{[0,s_\alpha(g)]})\). In order to finish it would be enough to prove that \(\mathrm{min}\{h'_0,h'_1\}=h'_1\); but this is immediate, as
$$
 h'_0-h'_1=\alpha\beta-a\alpha-\alpha\beta+(\alpha-b)\beta=\alpha\beta-a\alpha-b\beta=g>0.
$$
\end{proof}

The previous discussion together with Proposition \ref{prop:last} lead us to propose the following definition of supersymmetry of gaps for a numerical semigroup of arbitrary embedding dimension.

\begin{defin}\label{defin:extensionsym}
	Let \(\Gamma=\langle x_1,\dots,x_n\rangle\) be a numerical semigroup minimally generated by $x_1,\ldots, x_n$. We define on the set of gaps \(G:=\mathbb{N}\setminus\Gamma\) the relation
	\[
	g_1\sim_{c}g_2\Longleftrightarrow c(g_1)=c(g_2) \ \ \mbox{for any} \ g_1, g_2 \in G.
	\]
	This is in fact an equivalence relation which thus provides a partition of the set gaps into equivalence classes. This partition will be called the \textit{gap conductor partition} of $G$. We say that two gaps \(g_1,g_2\) are \textit{candidates to be supersymmetric} if \(g_1\sim_{c} g_2\). In addition, we say that two gaps \(g_1,g_2\) are \textit{supersymmetric} if \(g_1\sim_{c} g_2\) and \(|W(g_1)|=|W(g_2)|\) and we will say that \(g\) is \textit{self-symmetric} if there is no other \(g'\) such that \(g\sim_{c} g'\).
\end{defin}
\begin{rem}
Observe that in the case of a numerical semigroup with embedding dimension \(2\), the sets of supersymmetric gaps and self-symmetric gaps are the set of elements that represent the equivalence classes of a special set of supersymmetric and self-symmetric gaps in terms of the previous definition.
\end{rem}
\begin{rem}
Observe that we are giving a definition of two gaps to be supersymmetric. This definition has no relation and has not to be mixed with the symmetry properties of the semigroup in the sense of \cite{KH}.
\end{rem}

In the special case \(\Gamma=\langle\alpha,\beta\rangle\), two gaps \(g_1,g_2\) are supersymmetric if and only if \(s_\alpha(g_1)=g_2\) or \(s_\beta(g_1)=g_2\). Moreover, there is no three different supersymmetric gaps; i.e. either \(g_1\) is supersymmetric to a unique \(g_2\neq g_1\) or \(g_1\) is self-symmetric and then it is its own supersymmetric point. Therefore, Definition \ref{defin:extensionsym} allows us to extent the properties of the lattice symmetries to purely algebraic properties of the gaps. This discussion leads to pose the following closing questions:

\begin{question}
	Given a symmetric numerical semigroup \(\Gamma=\langle x_1,\dots,x_n\rangle\), we ask:
	\begin{enumerate}
		\item[(1)] For \(n>2\), does there exist a subset of the set of gaps such that the supersymmetry property defined in Definition \ref{defin:extensionsym} allows to recover the whole semigroup from this set?
		\item[(2)] Does there exist a lattice representation in \(\mathbb{Z}^{n}\) of the set of gaps such that ---in analogy with the case \(n=2\)--- it can be made up from the sets of self-symmetric and supersymmetric gaps together with some affine transformation of them?
	\end{enumerate}
\end{question}


\end{document}